\documentclass[a4paper,10pt,english,reqno]{amsart} 

\DeclareSymbolFontAlphabet{\mathcal}{symbols}

\usepackage{multirow}
\usepackage{multicol}

\usepackage{soul} 

\usepackage[T1]{fontenc}
\usepackage{tikz}
\usetikzlibrary{fadings}
\usepackage[a4paper]{geometry}
\usepackage{pgfplots}
\pgfplotsset{compat=1.13}
\usepackage{mathrsfs}
\usetikzlibrary{arrows}
\usepackage{framed}
\usepackage{color}
\usepackage{amstext}
\usepackage{amsthm}
\usepackage{amssymb}
\usepackage{varioref}
\usepackage{amssymb}
\usepackage{enumerate}
\usepackage[cmtip,all]{xy}
\usepackage{pifont}
\usepackage{graphicx}
\usepackage{url}
\usepackage{float}
\usepackage[shortlabels]{enumitem}
\usepackage[pagebackref=true]{hyperref}
\usepackage{hyperref}
\usepackage{backref}
\hypersetup{
    colorlinks,
    citecolor=blue,
    filecolor=black,
    linkcolor=blue,
    urlcolor=black
}

\usepackage{float}
\makeatletter
\@namedef{subjclassname@2020}{%
  \textup{2020} Mathematics Subject Classification}
\makeatother
\title[Intersection homology and homotopy groups]{Relation between intersection homology and homotopy groups}

\date{\today}

\author{David Chataur}
\address{Lamfa\\
Universit\'e de Picardie Jules Verne\\
33, rue Saint-Leu\\
80039 Amiens Cedex~1\\
         France}
\email{David.Chataur@u-picardie.fr}

\author{Martintxo Saralegi-Aranguren}
\address{Laboratoire de Math{\'e}matiques de Lens\\  
      EA 2462 \\
      Universit\'e d'Artois\\
         SP18, rue Jean Souvraz\\
          62307 Lens Cedex\\
         France}
\email{martin.saraleguiaranguren@univ-artois.fr}

\author{Daniel Tanr\'e}
\address{D\'epartement de Math{\'e}matiques\\
         UMR-CNRS 8524 \\
         Universit\'e de Lille\\
         59655 Villeneuve d'Ascq Cedex\\
         France}
\email{Daniel.Tanre@univ-lille.fr}

\thanks{The first author was supported by the research project ANR-18-CE93-0002  ``OCHOTO'' . 
The third author was partially supported by 
the Proyecto PID2020-114474GB-100
and 
the ANR-11-LABX-0007-01  ``CEMPI''}



\makeatletter
\renewcommand\l@subsection{\@tocline{2}{0pt}{2pc}{5pc}{}}
\renewcommand\l@subsubsection{\@tocline{3}{0pt}{4pc}{10pc}{}}
\makeatother

\theoremstyle{plain}
\newtheorem{theorem}{Theorem}

\newtheorem{conjecture}{Conjecture}

\newtheorem{proposition}{Proposition}[section]
\newtheorem{theoremb}[proposition]{Theorem}
\newtheorem{lemma}[proposition]{Lemma}
\newtheorem{corollary}[proposition]{Corollary}

\theoremstyle{definition}
\newtheorem{definition}[proposition]{Definition}
\newtheorem{example}[proposition]{Example}

\newtheorem{recall}[proposition]{Recall}

\theoremstyle{remark}
\newtheorem{remark}[proposition]{Remark}

\numberwithin{equation}{section}
\newcommand{\secref}[1]{Section~\ref{#1}}

\newcommand{\thmref}[1]{Theorem~\ref{#1}}
\newcommand{\propref}[1]{Proposition~\ref{#1}}
\newcommand{\lemref}[1]{Lemma~\ref{#1}}
\newcommand{\corref}[1]{Corollary~\ref{#1}}

\newcommand{\remref}[1]{Remark~\ref{#1}}

\newcommand{\exemref}[1]{Example~\ref{#1}}

\newcommand{\defref}[1]{Definition~\ref{#1}}

\def\R{{\mathbb R}}

\def\ov{\overline}

\def\cC{{\mathcal C}}

\def\cE{{\mathcal E}}

\def\cJ{{\mathcal J}}

\def\cO{{\mathcal O}}

\def\cS{{\mathcal S}}
\def\cU{{\mathcal U}}
\def\cV{{\mathcal V}}
\def\cW{{\mathcal W}}

\def\crG{{\mathscr G}}

\def\crX{{\mathscr X}}

\def\1{{\mathbf 1}}

\def\tc{{\mathtt c}}

\def\tu{{\mathtt u}}
\def\tv{{\mathtt v}}
\def\tw{{\mathtt w}}

\def\C{\mathbb{C}}

\def\N{\mathbb{N}}
\def\Q{\mathbb{Q}}
\def\R{\mathbb{R}}
\def\Z{\mathbb{Z}}

\def\RP{{\mathbb{R}{\rm P}}} 

\def\ab{{\mathrm{ab}}}

\def\id{{\rm id}}

\def\codim{{\rm codim\,}}

\def\sing{{\rm Sing}}

\def\sd{{\rm sd}}
\def\rc{{\mathring{\tc}}}

\def\menos{\backslash}

\newcounter{ejemplo}

\def\depth{{\mathrm{depth \,}}}

\newcounter{figura}

\def\top{{\mathbf{Top}}}

\def\top{{\mathrm{Top}}}
\def\Ab{{\mathrm{Ab}}}

\def\eps{\varepsilon}
 \newcommand{\inte}[1]{\mathring{#1}} 
\renewcommand\1{\hbox{\ding{192}}}

\begin{document} 
\begin{abstract} 
As Goresky and MacPherson intersection homology is not  the homology of a space, there is
no preferred candidate for intersection homotopy groups.
Here, they are defined as  the homotopy groups of a simplicial set  which P.~Gajer 
associates to a couple $(X,\ov{p})$ of a filtered space and a perversity.
We first establish some basic properties  for the intersection fundamental groups, as a Van Kampen theorem.

For general intersection homotopy groups on  Siebenmann CS sets, 
we prove a Hurewicz theorem between them and the Goresky and MacPherson intersection homology.
If the CS set  and its intrinsic stratification have the same regular part,
we establish the topological invariance of the $\overline{p}$-intersection homotopy groups.
Several examples justify the hypotheses made in the statements.
Finally, intersection homotopy groups also coincide with the homotopy groups of the 
topological space itself, for the top perversity  on a connected, normal Thom-Mather space. 
\end{abstract}


\maketitle

%

\section*{Introduction}

Poincar\'e duality is an extraordinary property of manifolds. Trivial examples show that this feature disappears
when the space in consideration presents some singularities, even in the case of an amalgamation of manifolds
of different dimensions as in a complex of manifolds of Whitney (\cite{MR19306}). 
Using an extra parameter $\ov{p}$ called perversity, Goresky and MacPherson \cite{GM1,GM2} 
define  new homologies depending on the choice of a perversity  and recreate a 
Poincar\'e duality between some spaces with singularities, as the pseudomanifolds.

\smallskip
Many invariants and structures of differential or algebraic topology also
 found their place in intersection homology or cohomology: 
 Morse theory, characteristic classes, Hodge theory, bordism, existence of cup products, foliations,...
(See \cite[Chapter 10]{FriedmanBook} for a  documented list.) 
From a homotopical point of view, the first observation is that  intersection homology
is not a homotopy invariant. 
On the other hand,   F.~Quinn (\cite{Qui}) has given a  presentation of some filtered spaces
which allows their study with homotopical tools. 
Here, what we have in mind is a notion of intersection homotopy groups 
which could be related to intersection homology, more or less as
 the homotopy and homology groups of a space are. 
 In this direction, the first question is:
 ``Let $X$ be a given filtered space and $\ov{p}$ be a perversity. 
 Does there exist a topological space $Y$ whose ordinary homology groups are the $\ov{p}$-intersection homology groups of $X$?''
As quoted by N. Habbegger  in the introduction of \cite{MR826504}, his study of Thom operations 
for intersection homology ``destroys the hope of calculating intersection homology 
as the ordinary homology of a suitable space, in general, since the Thom operations are natural.''

\smallskip
Nevertheless the well-established properties of the homology of a space encourage a search for substitutes, 
i.e. the search for spaces $I^{\ov{p}}X$, associated to a pseudomanifold $X$ and a
perversity $\ov{p}$, and which is   ``close'' to $X$  in a direction to be specified. 
This is the approach of M.~Banagl and we send the reader to \cite{MR2662593} 
for an explicit description of this procedure and its properties; we will not use them in this work.
 Briefly, the idea guiding Banagl's construction is the replacement of the singular links by 
 a truncation of their Moore space decomposition.
 In general, the homology of $I^{\overline{p}}X$ is not the Goresky and MacPherson intersection homology
 but satisfies a generalized Poincar\'e duality when $X$ is closed and oriented.
 The  difficulty lies in the construction of the spaces $I^{\ov{p}}X$.
 Their existence is established for pseudomanifolds with isolated singularities and 
 some cases of depth two pseudomanifolds in \cite{MR2662593,MR2928934,MR4243079}.
 In \cite{MR4170059}, their construction is carried out for arbitrary depth but with trivial link bundles, 
 which covers the case of toric varieties.
 
 \smallskip
Our approach is of a different spirit. We keep the Goresky and MacPherson intersection homology and
introduce intersection homotopy groups as homotopy groups of a space defined 
by Gajer in \cite{MR1404919,MR1489215}. 
To explain that, it is better to come back at the beginning of the story.
Given a filtered space $X$ and a perversity $\ov{p}$, a selection is made among the singular simplexes as follows.
A simplex $\sigma\colon \Delta\to X$ is  
\emph{$\ov{p}$-allowable}  
if, for each singular stratum $S$, the set $\sigma^{-1}S$ verifies
\begin{equation}\label{equa:admissible}
\dim \sigma^{-1}S\leq \dim\Delta-\codim S +\ov{p}(S).
\end{equation}
A $\ov{p}$-allowable chain is a linear combination of $\ov{p}$-allowable simplexes.
These definitions present major shortcomings. 
 \begin{enumerate}[i)]
 \item If $\xi$ is a $\ov{p}$-allowable chain, its boundary $\partial \xi$ is not necessary $\ov{p}$-allowable.
  \item If $\sigma$ is a $\ov{p}$-allowable simplex and $\partial_{i}$ is a face operator, 
 the simplex $\partial_{i}\sigma$ is not necessary $\ov{p}$-allowable.
 \end{enumerate}
 
 If we are interested by homology, we do a more restrictive choice at the level of chains: 
 a singular  chain $\xi$ is  of \emph{$\ov{p}$-intersection} 
 if $\xi$ and its boundary $\partial \xi$ are $\ov{p}$-allowable (\cite{GM1}).
The  homology of the complex of singular chains of $\ov{p}$-intersection, $C_{*}^{\ov{p}}(X;G)$,
is the \emph{$\ov{p}$-intersection homology}  $H_*^{\ov{p}}(X;G)$, with coefficients in an  $R$-module~$G$
 over a Dedekind ring $R$.

\smallskip
If we are interested by simplicial sets, 
we act on the simplexes themselves: a simplex $\sigma$ is \emph{$\ov{p}$-full} if $\sigma$ and all 
its iterated faces are $\ov{p}$-allowable. This is the approach of \cite{MR1404919}:
Gajer gets a  Kan simplicial set that we 
denote by $\crG_{\ov{p}}X$ and called \emph{the Gajer space.} 
The \emph{$\ov{p}$-intersection homotopy groups} are now defined as the homotopy groups of $\crG_{\ov{p}}X$
and denoted by $\pi^{\ov{p}}_{*}(X)$.
What we do, in the present paper, is the study of these homotopy groups and their relation with intersection homology. 
For the dimension of $\sigma^{-1}S$  in \eqref{equa:admissible},  we are dealing with the polyhedral dimension of \cite{MR1404919}
(see \defref{def:dimension}), revisited in \cite{CST8}.
In \cite{MR1404919}, Gajer falsely believed that the homology of $\crG_{\ov{p}}X$ is the
$\ov{p}$-intersection homology of $X$. 
So in \cite{MR1404919}, some properties as the existence of a Mayer-Vietoris sequence
for $H_{*}(\crG_{\ov{p}}X)$ with open subsets of $X$ appeared as a consequence of well-known properties of
intersection homology. 
Also we  work with topological filtered spaces and not PL-ones.
For these reasons, in \secref{sec:gajerspace},  we provide the proofs of basic properties, 
with  references to their first  occurrence if we find one.

\smallskip
It's time to clarify what we mean by perversity. In this work, a \emph{perversity} is a map from the set of strata of a filtered space with values in $\Z\cup\{\pm\infty\}$,
taking the value 0 on the regular strata.
In particular, any map $f\colon \N\to\Z$ such that $f(0)=0$ defines a perversity $\ov{p}$ 
 by $\ov{p}(S)=f(\codim S)$.  Such perversities are said \emph{codimensional}. 
 Among them are the original GM-perversities (\cite{GM1}) of Goresky and MacPherson, see \defref{def:perversite},
 and the top perversity defined by $\ov{t}(i)=i-2$, for $i\geq 2$. 

\smallskip
In the case of a filtered space $X$ and a perversity $\ov{p}$, we establish a \emph{Van Kampen theorem} for $\pi^{\ov{p}}_{1}(X)$ in \thmref{thm:VK}.
If $X$ is a connected normal Thom-Mather space with a finite number of strata and $\ov{t}$ is the top perversity, 
we deduce, in \corref{cor:thommather},
the isomorphisms, $\pi_{j}^{\ov{t}}(X)\cong \pi_{j}(X)$, for any $j$.
This result enhances the well-known homology result of intersection homology, $H_{*}^{\ov{t}}(X)\cong H_{*}(X)$.

\smallskip
Our main results concern the locally conical filtered spaces, introduced by Siebenmann in \cite{MR0319207}, 
and called CS sets, see \defref{def:csset}. 
One of their first delicate peculiarities 
is that links of points in the same stratum are 
not necessarily homeomorphic (\cite[Example 2.3.6]{FriedmanBook}). 
We know that these links have isomorphic intersection homology groups,
and we prove here that they also have isomorphic intersection homotopy groups, cf. \propref{prop:LinkUnique}.

\smallskip
In \exemref{ex:poincaresphere}, we point out that the double suspension of the Poincar\'e sphere 
is a counter-example to the topological invariance of $\ov{p}$-intersection fundamental groups.
In this example, some singular points of the CS set $X$ become regular in the intrinsic stratification
$X^*$ of $X$, introduced by King in \cite{MR800845}.
The next  result (see \thmref{thm:thepione} and \thmref{thm:homotopyinvariance}) shows that, 
except this generic case, the perverse homotopy groups are topological invariants. 

\begin{theorem}\label{thm:thepioneintro}
Let $X$ be a CS set without stratum of codimension 1, of intrinsic stratification $\nu\colon X\to X^*$ and
$\ov{p}$ be a Goresky and MacPherson perversity.  
We suppose that the regular parts of $X$ and $X^*$ coincide. 
Then the map $\nu$ induces  isomorphisms,
$\pi_{j}^{\ov{p}}(X,x)\cong \pi_{j}^{\ov{p}}(X^*,x)$,
for any regular point $x$ and any $j$.
\end{theorem}

We also establish a $\ov{p}$-intersection analog of the Hurewicz theorem 
(\thmref{thm:hurintersection}) between intersection homotopy groups and Goresky and MacPherson intersection homology groups.
 \exemref{exem:pashurewicz}   justifies the hypotheses put on the links.
 We call \emph{$\ov{p}$-intersection Hurewicz homomorphism}, 
 $$h^{\ov{p}}_{*}\colon \pi^{\ov{p}}_{*}(X,x)\to H_{*}^{\ov{p}}(X;\Z),$$
 the composition of the Hurewicz map for $\crG_{\ov{p}}X$ with the homomorphism 
 $H_{*}(\crG_{\ov{p}}X;\Z)\to H_{*}^{\ov{p}}(X;\Z)$ coming 
 from  the inclusion between the corresponding chain complexes.

\begin{theorem}\label{thm:HureIntersectionHomotopyHomologyintro}
Let $X$ be a  CS set and $\ov{p}$ be a perversity such that $\pi_{0}^{\ov{p}}(X)=0$.
\begin{enumerate}[i)]
\item For any regular point $x$, the $\ov{p}$-intersection Hurewicz map 
$h^{\ov{p}}_{1}\colon \pi^{\ov{p}}_{1}(X,x)\to \widetilde{H}_{1}^{\ov{p}}(X;\Z)$
induces an isomorphism between the abelianisation of $\pi_{1}^{\ov{p}}(X,x)$ and $\widetilde{H}_{1}^{\ov{p}}(X;\Z)$.
\item
Let $k\geq 2$. 
 We suppose  
$\pi_{j}^{\ov{p}}(X)=\pi_{j}^{\ov p}(L)=0$ for every link  $L$ of $X$, and each $j\leq k-1$.
Then, the intersection Hurewicz homomorphism
$h_{j}^{\ov{p}}\colon \pi_{j}^{\ov{p}}(X,x_{0})\to \widetilde{H}_{j}^{\ov{p}}(X;\Z)$
 is an isomorphism for $j\leq k$ and a surjection for $j=k+1$.
\end{enumerate}
\end{theorem}

Some natural questions on $\ov{p}$-intersection homotopy groups follow from these results. 
Let us list some of them.

\smallskip $\bullet$
In the case of a cone on a compact space, $\rc X$, of apex $\tv$, 
with the conical stratification, $\{\tv\}\subset \rc X$, one has (\propref{prop:coneandhomotopy})
$$\crG_{\ov{p}}  (\rc X)=P_{k}\sing X,$$
where $P_{k}$ denotes the $k$th space of the Postnikov decomposition of $\sing X$
and $k=D\ov{p}(\tv)$. In this particular case,
the Gajer space appears as an Eckmann-Hilton dual (\cite[Section 4.H]{MR1867354})
of Banagl's intersection space, $I^{\ov{p}}X$: 
the (co)truncation of the  Moore decomposition is replaced by the truncation of the Postnikov tower.
In \cite{CST10}, we  come back to this point of view  and extend this relation to more general spaces than cones.

\smallskip $\bullet$
By using an adaptated version of the PL forms of D. Sullivan, we present in  \cite{CST1} 
 a notion of perverse minimal model. In a future work, we will connect
the indecomposables of this model with the intersection homotopy groups.
This will bring a new notion of $\ov{p}$-formality for a fixed GM-perversity $\ov{p}$.

\smallskip $\bullet$
Let's extend this list with: 
``Can we relate the perverse Eilenberg-MacLane spaces introduced in \cite{CT1} with the
intersection homotopy groups?''
and ``Is there a relation to perverse sheaves?''

\smallskip $\bullet$
Let us conclude with a question about small resolutions, which has its source in \cite[Section 6.2]{GM2}.
Recall that a small resolution of a complex algebraic variety is an algebraic map,
$f\colon V\to X$,
which is a resolution of singularities of $X$ such that, for all $r>0$, 
$$\codim_{\C}\{x\in X\mid \dim_{\C }f^{-1}(x)\geq r\}>2r.$$
In \cite{GM2}, the authors prove that a small resolution $f$ induces  isomorphisms,
$f_{*}\colon H_{*}(V)\cong H_{*}^{\ov{m}}(X)$,
between the (ordinary) homology of $V$ and the intersection homology of $X$ for the middle perversity $\ov{m}$.
When they do exist,  small resolutions are not necessarily unique and the previous result implies that the homology 
groups $H_{*}(V)$ do not depend on the small resolution of $X$. Thus we ask if there are similar results
for the intersection homotopy groups. 
For instance, under some connectedness conditions, can we find an integer $k$ such that
the homotopy groups $\pi_{j}(V)$ do not depend on the small resolution $V$, for $j\leq k$. 
More precisely, we make the following conjecture.
\begin{conjecture}
Let $f\colon V\to X$ be a small resolution. Suppose that $V$ and $X$ are simply connected and 
that there exists an integer $k$ such that $\pi_{j}(L)=0$ for all links and all $j<k$. Then, we have an isomorphism
$$f_{*}\colon \pi_{j}(V)\to \pi_{j}^{\ov{m}}(X), \quad\text{for all} \quad j\leq k.$$
\end{conjecture}
\noindent
In the case of isolated singular points of a local complete intersection of complex dimension $n$, this 
 conjecture implies the existence of isomorphisms, $\pi_{j}(V)\cong \pi_{j}^{\ov{m}}(X)$, 
 for any $j\leq n-1$ and any small resolution
$V$ of $X$.

\medskip
\paragraph{{\bf Notation and convention.}}
The word ``space'' means compactly generated topological space.
When we work on CS sets that are locally path-connected spaces, we use ``connected''  and  ``path-connected'' interchangeably.
In the text, the letter $R$ denotes a Dedekind ring and $G$ an $R$-module. 
The singular chain complex of a  simplicial set, $X$,  is denoted by
$C_*(X;G)$, $C_{*}(X;\cE)$ or $C_{*}(X)$ if there is no ambiguity.

We denote by $\Delta[\ell]$ the simplicial set whose $k$-simplexes are the $(k+1)$-uple of integers
$(j_{0},\dots,j_{k})$ with $0\leq j_{0}\leq\dots\leq j_{k}\leq \ell$, and by
$d_{i}\colon \Delta[\ell]_{k}\to \Delta[\ell]_{k-1}$,
$s_{i}\colon \Delta[\ell]_{k}\to \Delta[\ell]_{k+1}$,
its faces and degeneracies, for $i\in\{0,\dots,k\}$.
Its geometric realization  is the  subspace of $\R^{\ell+1}$
defined by
$\Delta^{\ell}=\{(t_{0},\dots,t_{\ell})\mid \sum_{i=0}^{\ell}t_{i}=1, \,t_{i}\geq 0\}$.
We denote by $\inte\Delta^\ell=\Delta^\ell\menos \partial \Delta^\ell$ the interior of $\Delta^\ell$.
The $k$th horn, $\Lambda[\ell,k]$, of  $\Delta[\ell]$ is obtained from the boundary 
$\partial \Delta[\ell]$ by removing the $k$th face.
 Its geometric realization is denoted by $\Lambda^{\ell}_{k}$.

The family $\Delta^{\bullet}$ is a cosimplicial space with cofaces and codegeneracies,
$d^i\colon \Delta^{\ell-1}\to \Delta^{\ell}$ and
$s^i\colon \Delta^{\ell+1}\to \Delta^{\ell}$,
for $i\in\{0,\dots,\ell\}$.
The image $d^i\Delta^{\ell-1}$ is called the \emph{$i$-face} of $\Delta^\ell$ and denoted by $\partial_{i}\Delta^\ell$.
If $\sigma\colon \Delta^{\ell}\to X$ is a singular simplex  
the precomposition of $\sigma$ with $d^i$ is denoted by
$\partial_{i}\sigma=\sigma\circ d^i$.
The domain of  $\sigma$ is denoted by
 $\Delta^\ell$ 
or  $\Delta_{\sigma}$, or simply $\Delta$,
depending on the parameter concerned at this place.

%
Our program is carried out in Sections 1-6 below, whose headings are self- explanatory.
\tableofcontents

%
\centerline{\textsc{ Acknowledgments}}

\medskip
The authors are grateful to the referees for their valuable comments and suggestions.

%

\section{Background on filtered spaces and intersection homology}


The singular spaces considered by Goresky and MacPherson in their pioneer works
 (\cite{GM1,GM2}), are  \emph{pseudomanifolds}.
 They allow a Poincar\'e duality through intersection homology, based on a parameter called perversity.
The singular spaces presented in this section are general filtered spaces and the  CS sets  of Siebenmann 
(\cite{MR0319207}) well adaptated to proofs by induction. 
We also make use of  perversities defined  on the poset of strata.

\subsection{Filtered spaces and CS sets}\label{sec:pseudo}
 We denote by $\top$ the category of compactly generated spaces (henceforth called ``space'') 
 with arrows the continuous maps.
Let us  emphasize that our definition of compact space includes the Hausdorff property. We now enter in the filtered world.

\begin{definition}\label{def:filteredspace}
A \emph{filtered space of formal dimension $n$} is a non-empty space $X$ 
with a filtration,
$$
X_{-1}=\emptyset\subseteq X_{0} \subseteq X_{1}\subseteq \dots \subseteq X_{n-2} \subseteq X_{n-1} \subsetneq X_{n} =X,
$$
by closed subsets. 
The subspaces $X_{i}$ are called \emph{skeleta} of the filtration and the 
non-empty components of $X^{i}=X_{i}\menos X_{i-1}$ are the \emph{strata}  of $X$. 
The set of strata is denoted by $\cS_{X}$ (or $\cS$ if there is no ambiguity).
The subspace $\Sigma_{X}=X_{n-1}$ (or $\Sigma$)  is called the \emph{singular subspace.}
Its complementary, $X\menos \Sigma_{X}$, is the \emph{regular subspace} and the
strata in $X\menos \Sigma_{X}$ are called \emph{regular.} 
The formal dimension of a stratum $S\subset X^{i}$ is $\dim S=i$; its formal codimension is
$\codim S=n-i$.
(The formal dimension is not necessarily related to  any topological notion of dimension.)
The filtered space is of \emph{locally finite stratification} if every point has a neighborhood that intersects only a finite number of strata.

A \emph{stratified space} is a filtered space such that the closure of any stratum is a union of strata of lower dimension.
The set of strata of a stratified space is a poset, $(\cS,\prec)$, for the relation
$S_{0}\prec S_{1}$ if $S_{0}\subset \ov{S}_{1}$.
The \emph{depth of $X$}    is the greatest integer $m$ for which it exists a chain of strata,
 $S_{0}\prec S_{1}\prec \cdots\prec S_{m} $.
 We denote it by $\depth X=m$.
\end{definition}

Let $X$ be a filtered space of formal dimension $n$.
An open subset $U \subset X$ is a filtered space for the \emph{induced filtration} given by
$U_i = U \cap X_{i}$.
The product $Y\times X$ with  a  space $Y$  is a filtered space for the \emph{product filtration} defined by
$\left(Y \times X\right) _i = Y \times X_{i}$. 
 If $X$ is compact, the open cone 
$\rc X = X \times [0,1[ \big/ X \times \{ 0 \}$ is endowed with the 
\emph{conical filtration} defined by
$\left(\rc X\right) _i =\rc X_{i-1}$,  $0\leq i\leq n+1$. 
By convention, 
$\rc \,\emptyset=\{ \tv \}$, where $\tv=[-,0]$ is the apex of the cone.

\begin{definition}\label{def:startifiedmap}
A \emph{stratified map} $f\colon X\to Y$ between two filtered spaces is a continuous map such that
for each stratum $S$ of $X$ there exists a stratum $S^f$ of $Y$ with $f(S)\subset S^f$
and $\codim S^f\leq \codim S$.
A \emph{stratified homeomorphism} is a homeomorphism such that $f$ and $f^{-1}$ are stratified maps.
We denote this relation by $X\cong_{s} Y$.
\end{definition}

In this text, we will frequently encounter  stratified maps $f$ which are homeomorphisms but for which the reverse application 
$f^{-1}$ is not stratified. We call them \emph{homeomorphisms and stratified maps.}  
It is fundamental not to confuse them with stratified homeomorphisms.

\begin{definition}\label{def:homotopyequistrat}
Let $X,\,Y$ be filtered spaces, we endow the cylinder $X\times [0,1]$ with the product filtration,
$(X\times [0,1])_{i}=X_{i}\times [0,1]$. 
Two stratified maps, $f,\,g\colon X\to Y$, are \emph{stratified homotopic} if there exists a stratified map 
$H\colon X\times [0,1]\to Y$ such that
$f=H(-,0)$ and $g=H(-,1)$. We denote this relation by $f\simeq_{s} g$.
Two filtered spaces, $X$ and $Y$, are \emph{stratified homotopy equivalent} 
if there exist stratified maps 
$f\colon X\to Y$ and $g\colon Y\to X$ such that
$g\circ f \simeq_{s} \id_{X}$, $f\circ g \simeq_{s} \id_{Y}$ and 
if $\codim S=\codim S^f$, $\codim T=\codim T^g$,
for any strata $S$ of $X$ and $T$ of $Y$.
We denote this relation by $X \simeq_{s} Y$ and the maps $f$ and $g$ are called 
\emph{stratified homotopy equivalences.}
\end{definition}
A stratified homotopy equivalence, $f\colon X\to Y$, induces a bijection between the two sets of strata, 
$\cS_{X}\cong \cS_{Y}$, and verifies $S^{g\circ f}=S$, $T^{f\circ g}=T$
see \cite[Remark 2.9.11]{FriedmanBook}.
Let us end these recalls with the CS sets, introduced by  Siebenmann in \cite{MR0319207}.

\begin{definition}\label{def:csset}
A \emph{CS set} of dimension $n$ is a filtered space $X$ of dimension $n$,
whose $i$-dimensional strata are $i$-dimensional topological manifolds for each $i$, and
such that for each point $x \in X_i \backslash X_{i-1}$, $i\neq n$, there exist
 an open neighborhood $V$ of  $x$ in $X$, endowed with the induced filtration,
an open neighborhood $U$ of $x$ in $X_i\backslash X_{i-1}$, 
a compact filtered space, $L$, of dimension $n-i-1$,
and a stratified homeomorphism $\varphi \colon U \times \rc L\to V$ such that
\begin{equation}\label{equa:filteredlink}
\varphi(U\times \rc L_{j})=V\cap X_{i+j+1},
\end{equation}
 for each $j\in \{0,\dots,n-i-1\}$.
 If $\tv$ is the apex of the cone $\rc L$, the homeomorphism
 $\varphi$ is also required to verify $\varphi(u,\tv)=u$, for each $u\in U$.
The pair  $(V,\varphi)$ is a   \emph{conical chart} of $x$
 and the filtered space $L$ is a \emph{link} of $x$.
 The CS set $X$ is called  \emph{normal} if its links are connected.
\end{definition}

From the definition, we  see that a CS set has a topological dimension. Moreover, in a CS set, the notions of formal and topological dimension coincide.

\begin{remark}\label{rem:CSmanifold}
If $X$ is a manifold with a structure of CS set, then any link of $X$ is a homology sphere. 
So, if the CS set $X$ has no codimensional 1 strata, it is normal.
Let us see that. 

Given a link $L$ we have a stratified homeomorphism $\varphi \colon \R^k \times \rc L \to V$, 
where $V$ is an open subset of $X$. Without loss of generality, we can suppose that $V$ is homeomorphic to 
an open subset of $\R^n$. 
So, we have a topological embedding $\psi \colon  \R^k \times \rc L=\rc (S^{k-1}*L)  \to \R^n = \rc S^{n-1}$ 
preserving the apexes. Following Stalling's invertible cobordism, we get a homeomorphism $ \rc (S^{k-1}*L) = \rc S^{n-1}$
 preserving the apexes (\cite[Proposition 1]{MR800845}) and thus $L$ is a homology sphere.
\end{remark}

There can be some differences in the various definitions of CS sets in the literature. 
In \defref{def:csset}, the links of singular strata are supposed to be non-empty, 
 thus the open subset $ X_ {n} \backslash X_ {n-1} $ is dense.  
 This implies that, for each link $L$, the regular part, $L\menos \Sigma_{L}$, is dense in $L$.
 The links of a stratum are not uniquely defined but if  one of them is connected, so all of them are too
(\cite[Remark 2.6.2]{FriedmanBook}).
A CS set   is a stratified space (\cite[Appendix A2]{CST1}) and its strata are locally closed 
and form a locally finite family (\cite[Lemma 2.3.8]{FriedmanBook}).

\begin{remark}\label{rem:conexenormal}
Proceeding as in \cite[Lemma 2.6.3]{FriedmanBook}, one can prove that
 a normal CS set $X$ is connected if, and only if, its regular part, $X\menos \Sigma$, is connected. We can also
 prove that $X$ is normal if, and only if, the regular part $L\menos \Sigma_L$ of any link $L$ is connected.
 \end{remark}

%

\subsection{Perversities}\label{sec:perversity}
Intersection homology of Goresky and MacPherson is defined from a parameter, called perversity. 
The original ones introduced in \cite{GM1} depend only on the codimension of the strata.
More general perversities  (\cite{MacPherson90,FriedmanBook, MR1245833, MR2210257}) 
are defined on the set of strata. 
In \cite{CST4}, a  blown-up cohomology is defined in this setting, to establish a
Poincar\'e duality for pseudomanifolds with a cap product in \cite{CST2,ST1}.
Let us recall their definitions.

\begin{definition}\label{def:perversitegen}
A \emph{perversity on a filtered space, $X$,} is a map $\ov{p}\colon \cS_{X}\to \ov{\Z}={\Z}\cup \{\pm\infty\}$
taking the value 0 on the regular strata. The pair $(X,\ov{p})$
is called a \emph{perverse space}, or a \emph{perverse CS set} if $X$ is a CS set. 

A \emph{constant perversity} $\ov k$, with $k\in \ov \Z$, is defined by $\ov k(S)=k$ for any singular stratum $S$.
The \emph{top perversity}  $\ov t$ is
 defined by $\ov{t}(S)=\codim S-2$, if $S$ is a singular stratum. 
 Given a perversity $\ov p$ on $X$,  the \emph{complementary perversity} on $X$, $D\ov{p}$, is characterized 
 by $D\ov{p}+\ov{p}=\ov{t}$.

\smallskip
Any map $f\colon \N\to\Z$ such that $f(0)=0$ defines a perversity $\ov{p}$ 
 by $\ov{p}(S)=f(\codim S)$. Such perversity is called \emph{codimensional}. 
In general, we denote by the same letter the perversity $\ov{p}$ and the map $f$.
Among the codimensional perversities we find the original perversities of  \cite{GM1}.
\end{definition}

\begin{definition}\label{def:perversite}
A \emph{Goresky and MacPherson perversity (or GM-perversity)} is a map 
$\ov{p}\colon \{2,\,3,\dots\}\to\N$ such that  
$\ov{p}(2)=0$ and
$\ov{p}(i)\leq\ov{p}(i+1)\leq\ov{p}(i)+1$ for all $i\geq 2$. 
\end{definition}
When using a GM-perversity, the filtered spaces under consideration have no 1-codimensional strata.
If $\ov{p}$ is a GM-perversity, so is its complementary.

\begin{definition}\label{def:backperversity}
Let $f\colon X\to Y$ be a stratified map and $\ov{p}$ be a perversity on $Y$. The \emph{pullback perversity of $\ov{p}$ by $f$}
is the perversity $f^*\ov{p}$ on $X$ defined by $f^*\ov{p}(S)=\ov{p}(S^f)$ for any stratum $S$ of $X$.
In the case of the canonical injection of an open subset endowed with the induced filtration, $\iota\colon U\to Y$, 
we still denote by $\ov{p}$ the perversity $\iota^*\ov{p}$ and call it the \emph{induced perversity.}
\end{definition}

The  product with a topological space, $X\times M$, is endowed with
the pull-back perversity of $\ov{p}$ by the canonical projection $X\times M\to X$, also denoted by $\ov{p}$. 
If $X$ is compact, a perversity $\ov{q}$ on the open cone  $\rc X$ induces a perversity on $X$,
also denoted by $\ov{q}$ and defined by
$\ov{q}(S)=\ov{q}(S\times ]0,1[)$. 

\begin{remark}\label{rem:perverselink}
A perversity $\ov{p}$ defined on a CS set, $X$,  induces a perversity on any link $L$.
Let $ \varphi\colon \R^k\times \rc L\to V$ be a conical chart of $x\in S\in \cS_{X}$.
Then, $V\menos S$ is an open subset of $X$ and is given with the induced perversity still denoted by $\ov{p}$. 
The strata of $V\menos S$ are the products $\R^k\times ]0,1[\times T$, where $T$ is a stratum of $L$. We set
$\ov{p}(T)=\ov{p}(\R^k\times [0,1[\times T)$ which defines a perversity on $L$.
There is also a relation between the perversity on $S$ itself and the conical chart by setting $\ov{p}(\tv)=\ov{p}(S)$.
\end{remark}

\subsection{Intersection homology}\label{sec:lesdeux}

Let $(X,\ov{p})$ be an $n$-dimensional perverse space.
The starting point  in intersection homology is the use of the perversity  for making
a selection among singular simplexes. 
 Before stating it, we   specify the notion of dimension we are using 
 for  subspaces of a Euclidean simplex. 
 We employ the polyhedra containing it (see \cite{MR0350744} for the used properties on polyhedra).

\begin{definition}\label{def:dimension}
A subspace $A\subset \Delta$ of a Euclidean simplex is of \emph{polyhedral dimension} less than or equal to $\ell$
 if  $A$ is included in a polyhedron $Q$ with $\dim Q\leq \ell$. This definition verifies
 \begin{equation}\label{equa:dimunion} 
 \dim (A_{1}\cup A_{2})= \max (\dim A_{1},\dim A_{2}).
 \end{equation}
\end{definition}

We choose this  definition and do   a selection among singular simplexes. 
As it appears below, the allowability condition is usually most conveniently expressed in terms of the
complementary perversity, $D\ov{p}$, rather than $\ov{p}$.

\begin{definition}\label{def:homotopygeom}
Let $(X,\ov{p})$  be a  perverse space.
A simplex $\sigma\colon \Delta\to X$ is  
\emph{$\ov{p}$-allowable}  
if, for each singular stratum $S$, the set $\sigma^{-1}S$ verifies
\begin{equation}\label{equa:admissibleST}
\dim \sigma^{-1}S\leq \dim\Delta-\codim S +\ov{p}(S)=\dim\Delta-2-D\ov{p}(S),
\end{equation}
with the convention $\dim\emptyset=-\infty$.
A singular  chain $\xi$ is \emph{$\ov{p}$-allowable} if it can be written as a linear combination of 
$\ov{p}$-allowable  simplexes.
\end{definition}

 In \remref{rem:pseusobarycenter}, we illustrate this definition  in low dimensions with a  perversity such that 
 $D\ov{p}\geq 0$, which is the case of the GM-perversities.

\medskip
As explained in the introduction, we need a more restrictive choice than $\ov{p}$-allowability 
for obtaining a chain complex.
 
\begin{definition}\label{def:intersectionhomology}
A singular  chain $\xi$ is  of \emph{$\ov{p}$-intersection} if $\xi$ and its boundary $\partial \xi$ are $\ov{p}$-allowable.
We denote by $C_{*}^{\ov{p}}(X;G)$ the complex of singular chains of $\ov{p}$-intersection
and by $H_*^{\ov{p}}(X;G)$ its homology, called
\emph{$\ov{p}$-intersection homology of $X$ with coefficients in a module $G$ over a Dedekind ring $R$}.
\end{definition}

If $f\colon (X,\ov{p})\to (Y,\ov{q})$ is  a stratified map between two perverse  spaces 
with  $f^*D\ov q \leq D\ov p$,   the association $\sigma\mapsto f\circ \sigma$ sends 
a $\ov{p}$-allowable simplex on a $\ov{q}$-allowable simplex
(\cite[Proposition 1]{CST8}) and defines a chain map
$f_{*}\colon C_*^{\ov p} (X)\to C_*^{\ov q} (Y)$.

\begin{remark}\label{rem:pseusobarycenter}
In the original definition of King (\cite{MR800845}), the  dimension chosen for a subspace of a Euclidean simplex 
comes from the dimension of the skeleta containing it. In \cite{CST8}, we show that the 
two choices of dimension
 give
 isomorphic intersection homology for CS sets. 
 The proof needs a Mayer-Vietoris exact sequence and therefore a small chains replacement.
 As usual, this property comes from a subdivision process but, with the choice of the polyhedra dimension
 in \defref{def:dimension}, that is a subtle issue.
 
 \smallskip
 Let us take basic examples with a perversity $\ov{p}$, such that $D\ov{p}\geq 0$.
 A 0-simplex and a 1-simplex are $\ov{p}$-allowable if, and only if, $\sigma^{-1}S=\emptyset$ for any singular stratum $S$:
this means that they must lie in the regular part.
For a 2-simplex, the situation is different: the set $\sigma^{-1}S$ can be a finite subset.
Let us consider a 2-simplex $\Delta^2$ with one singular point, $v$, in its interior, which is thus $\ov{p}$-allowable.
Now we want to subdivise $\Delta^2$ while keeping the $\ov{p}$-allowability condition for the small 2-simplexes of the
subdivision.
To achieve this, the point $v$ must not belong to any 1-simplex of the subdivision.

\medskip
Therefore, the subdivision of a $\ov{p}$-allowable simplex $\Delta^\ell$, 
that we need, cannot be the barycentric subdivision but 
a subdivision, called  \emph{pseudobarycentric subdivision,} which verifies the following properties:
\begin{itemize}
\item a diameter of the new simplexes strictly smaller than that of the initial simplex,
\item the preservation of $\ov{p}$-allowability,
\item a construction by induction, taking the cone with vertex a suitably chosen interior point $b$ in  $\Delta^\ell$
and base the pseudobarycentric subdivision of the boundary of $\Delta^\ell$. 
The point $b$ is called a \emph{pseudobarycentre} of $\Delta^\ell$.
\end{itemize}
In \cite[Proposition 6]{CST8}, we show that such subdivisions exist for any $\ov{p}$-allowable simplex.
In this work we only use the properties recalled above.
 \end{remark}
 %
 %

%
\section{Gajer spaces}\label{sec:gajerspace}

 Let $(X,\ov{p})$ be an $n$-dimensional perverse space.
We emphasize that in the definition of intersection homology (see \secref{sec:lesdeux}), 
we do not consider simplexes but chains, $\xi$, and 
we ask for the allowability of $\xi$ and its boundary $\partial \xi$. 
Now we want to construct a simplicial set,
thus the requirements concern the simplexes and all their faces.  

\begin{definition}\label{def:gajersimplicialset}
 Let $(X,\ov{p})$ be a  perverse space.
 A simplex $\sigma\colon \Delta^{\ell}\to X$ is  \emph{$\ov{p}$-full}  if
$\sigma$ and all its faces, $\partial_{i_{1}} \dots \partial_{i_{k}}\sigma$,
 are $\ov{p}$-allowable.
 \end{definition}

\begin{remark}
Similarly,  a simplex $\sigma\colon \Delta^{\ell}\to X$ is  \emph{$\ov{p}$-full} if,
for all faces $F$ of $\Delta$ and all singular stratum $S$, one has
$\dim(\sigma^{-1}S\cap F)\leq \dim F-D\ov{p}(S)-2$.
\end{remark}

We continue with basic properties. 
Let us begin with the following result, crucial for the definition of $\ov{p}$-intersection homotopy groups.

\begin{proposition}[{\cite[Page 946]{MR1404919}}]\label{prop:gajerkan}
Let $(X,\ov{p})$ be a perverse space. Then the set of $\ov{p}$-full simplexes is a simplicial set
verifying the Kan condition. We  denote it by $\crG_{\ov{p}}X$ and call it
 the \emph{Gajer $\ov{p}$-space} associated to $X$.
\end{proposition}

\begin{proof} %
Let $\sigma\colon \Delta^\ell\to X$ be a $\ov{p}$-full simplex and
$S$ be a singular stratum.  By definition, the face $\partial_{i}\sigma$ is  $\ov{p}$-full  for any face of $\sigma$.
 Let $s^i\colon \Delta^{\ell+1}\to \Delta^\ell$ be a codegeneracy and $d^j\colon\Delta^\ell \to \Delta^{\ell+1}$
be a coface. The commutator rules %
express
$\sigma\circ s^i\circ d^j$ as a composition $\sigma\circ d^r\circ s^t$. By induction on the dimension~$\ell$, we 
deduce that $\sigma\circ s^i$ is $\ov{p}$-full.
Thus the set of $\ov{p}$-full simplexes is a simplicial subset of the singular set $\sing \,X$.
By using the adjunction between the functor $\sing$ and the realization functor, 
the Kan extension condition is equivalent to the construction of
a $\ov{p}$-full simplex $\tau$ making commutative the following diagram,
$$\xymatrix{
\Lambda^\ell_{k}\ar[r]^-{|\Psi |}\ar@{^(->}[d]&X,\\
\Delta^\ell\ar@{-->}[ru]_-{\tau}&
}$$
for any $0\leq k\leq \ell$.
If $\rho_{k}\colon \Delta^\ell\to \Lambda^\ell_{k}$ denotes the radial projection 
from the barycentre of the face $\partial_{k}\Delta^\ell$, we set
$\tau=|\Psi|\circ \rho_{k}$.
For any singular stratum $S$, we have
$$\dim \left(|\Psi|\circ \rho_{k}\right)^{-1}S \leq 
\dim (|\Psi|)^{-1}S+1 \leq (\ell-1-D\ov{p}(S)-2)+1=\ell -D\ov{p}(S)-2,
$$
and the $\ov{p}$-allowability condition of $\tau$ is satisfied.
As $\rho_{k}$ is an affine map, the same argument works for any face of $\Delta^\ell$.
\end{proof}

The Gajer space depends on a filtration $\crX$ on the topological space $X$ and we
should denote it by $\crG_{\ov{p}}(X,\crX)$.
In fact,  we do not mention explicitly the filtration  and  simply write $\crG_{\ov{p}}X$.

\begin{proposition}{\cite[Example 1]{MR1404919}}\label{prop:homotopyxR}
 Let $(X,\ov{p})$ be a  perverse space 
  and $Y$ be a  topological space. 
 The product $Y\times X$ is equipped with the product filtration
 and the product perversity, also denoted by $\ov{p}$.  
 Then, the canonical projections, $p_{Y}\colon Y\times X\to Y$ and $p_{X}\colon Y\times X\to X$,
 induce an isomorphism 
 $$\crG_{\ov{p}}(Y\times X)\cong \sing \,Y\times \crG_{\ov{p}}X.$$
\end{proposition}

\begin{proof}
Let $\sigma=(\sigma_{Y},\sigma_{X})\colon \Delta^\ell\to Y\times X$ be a simplex.
By definition, $\sigma$ is $\ov{p}$-full if, and only if, $\sigma_{X}$ is $\ov{p}$-full also.
The projections therefore induce an isomorphism
$\crG_{\ov{p}}(Y\times X)\cong \sing \,Y\times \crG_{\ov{p}}X$.
\end{proof}

The following statement concerns the functoriality of the association
$(X,\ov{p})\mapsto \crG_{\ov{p}}X$
and its compatibility with homotopy.

\begin{proposition}\label{prop:gajermap}
Let $f \colon (X,\ov{p}) \to (Y,\ov{q})$ be a stratified map between perverse spaces such that
$f^*D\ov{q}\leq D\ov{p}$. 
Then, 
$f$ induces a map $\crG_{\ov{p},\ov{q}}f\colon \crG_{\ov{p}}X\to \crG_{\ov{q}}Y$.
Moreover, if  $\varphi\colon (X\times [0,1],\ov{p})\to (Y,\ov{q})$ is a  stratified homotopy between two stratified maps
$f,\,g\colon (X,\ov{p})\to (Y,\ov{q})$ with $f^*D\ov{q}\leq D\ov{p}$, then we also have $g^*D\ov{q}\leq D\ov{p}$
and the simplicial maps, $\crG_{\ov{p},\ov{q}}f$ and $\crG_{\ov{p},\ov{q}}g$ are homotopic.
\end{proposition}

\begin{proof}
As we have already noticed, the association
$\sigma\mapsto f\circ \sigma$ 
sends a $\ov{p}$-allowable simplex on a $\ov{q}$-allowable simplex.
Let $S$ be a stratum of $X$. Recall that the perversity on $X\times [0,1]$ is still denoted by $\ov{p}$ and defined by $\ov{p}(S\times [0,1])=\ov{p}(S)$.
From the stratification of the product $X\times [0,1]$, we have
$S^f=(S\times [0,1])^{\varphi}=S^g$. 
Thus, 
$f^*D\ov{q}(S)=D\ov{q}(S^f)=D\ov{q}((S\times [0,1])^{\varphi})=\varphi^*D\ov{q}(S\times [0,1])$
and, in the same way,
$g^*D\ov{q}(S)=\varphi^*D\ov{q}(S\times [0,1])$.
Thus the three conditions $f^*D\ov{q}\leq D\ov{p}$, $g^*D\ov{q}\leq D\ov{p}$ and $\varphi^*D\ov{q}\leq D\ov{p}$ are equivalent and verified by hypothesis.
The associations
$\sigma\mapsto f\circ \sigma, \;g\circ \sigma, \;\varphi\circ\sigma$
being simplicial maps, 
the maps $f$, $g$ and the homotopy $\varphi$ induce simplicial maps 
$\crG_{\ov{p},\ov{q}}f,\;\crG_{\ov{p},\ov{q}}g\colon\crG_{\ov{p}}X\to \crG_{\ov{q}}Y$
and
$\crG_{\ov{p},\ov{q}}\varphi\colon \crG_{\ov{p}}(X\times [0,1])\to \crG_{\ov{q}}Y$.
Using  \propref{prop:homotopyxR}, we get the desired homotopy as the composition
$$
\crG_{\ov{p}}X\times \Delta[1]\to \crG_{\ov{p}}X\times \sing(\Delta^1)\to
\crG_{\ov{p}}(X\times \Delta^1)
\xrightarrow{\;\crG_{\ov{p},\ov{q}\;}\varphi}
\crG_{\ov{q}}Y.
$$
\end{proof}

In particular, \defref{def:homotopyequistrat} and \propref{prop:gajermap} imply the following result.

\begin{corollary}\label{cor:gajerhomotopy}
Let $f\colon (X,\ov{p})\to (Y,\ov{q})$ be a stratified homotopy equivalence with $D\ov{p}=f^*D\ov{q}$.
Then $\crG_{\ov{p},\ov{q}}f\colon \crG_{\ov{p}}X\to \crG_{\ov{q}}Y$ is a homotopy equivalence.
\end{corollary}

\begin{proof}
Let $S$ be a stratum of $X$. By hypothesis,   there exists a stratified map $g\colon (Y,\ov{q})\to (X,\ov{p})$ and a stratified homotopy $\varphi$ between 
$g\circ f$ and $\id_{X}$.
From  the stratification of the product $X\times [0,1]$, we deduce
$S^{g\circ f}=(S\times [0,1])^{\varphi}=S^{\id}=S$.
Similarly, for any stratum $T$ of $Y$, we have $T^{f\circ g}=T$.
These equalities and the hypothesis $D\ov{p}=f^*D\ov{q}$ give $g^*D\ov{p}=D\ov{q}$. 
With \propref{prop:gajermap}, we get two simplicial maps,
$\crG_{\ov{p},\ov{q}}f\colon \crG_{\ov{p}}X\to \crG_{\ov{q}}Y$
and
$\crG_{\ov{q},\ov{p}}g\colon \crG_{\ov{q}}Y\to \crG_{\ov{p}}X$
and homotopies between their compositions and the identity maps.
\end{proof}

Let us  look at the particular case of isolated singular points.

\begin{proposition}\label{prop:isolatedgajer}
Let $(X,\ov{p})$ be a perverse space with an isolated singularity $x$.
Set $\ell=D\ov{p}(\{x\})+1$.
Then, there is an isomorphism between the $\ell$-skeleta of $\crG_{\ov{p}}X$ 
and $\crG_{\ov{p}}(X\menos \{x\})$. 
\end{proposition}

\begin{proof}
The $\ov{p}$-allowability condition of a simplex $\sigma\colon \Delta^j \to  X$ for a singular stratum $\{x\}$  is:
\begin{equation}\label{equa:encoreadmisv}
\dim \sigma^{-1}x \leq j - D\ov{p}(\{x\})-2 .
\end{equation}
Thus, for any $j\leq D\ov{p}(\{x\})+1$, the inequality \eqref{equa:encoreadmisv}  gives
$\dim \sigma^{-1}x \leq -1$.
This implies $\sigma^{-1}x=\emptyset$ and an isomorphism of the $j$-skeleta
$(\crG_{\ov{p}}X))^{j}\cong (\crG_{\ov{p}}(X\menos\{x\})^{j}$.
\end{proof}

Let $K$ be a simplicial set and $\Pi K$ be its fundamental groupoid. Let us recall that a
local coefficient system (of abelian groups) on $K$ is a contravariant functor
$\cE\colon \Pi K\to \Ab$,
with values in the category of abelian groups, see \cite[Appendix I]{MR0210125} or \cite[Page 340]{MR1711612}.
Such functor generates a chain complex $C_{*}(K;\cE)$ whose homology is called the homology of $K$
with coefficients in the local system $\cE$. The elements of $C_{j}(K;\cE)$ are chains
$\xi=\sum_{i\in I}a_{i}\sigma_{i}$ with 
$\sigma_{i}\in K_{j}$ 
and $a_{i}\in \cE(\sigma_{i}(0))$. 
The differential of  $a\sigma$, with $\sigma\in K_{j}$ and $a\in \cE(\sigma(0))$, 
is given by
$$\partial (a\sigma)=\sum_{i=1}^j (-1)^j a\,\partial_{i}\sigma+\cE(\sigma([01]))^{-1}(a)\,\partial_{0}\sigma.
$$
Isomorphisms between homotopy groups of simplicial sets (or spaces) can be detected from the existence of 
isomorphisms between fundamental groups and isomorphisms of homology groups
for any local coefficient system 
 (\cite[Proposition II.3.4]{MR0223432}). 

\begin{lemma}\label{lem:quillen}
For a simplicial map $f\colon K\to K'$, the following conditions are equivalent. 
\begin{enumerate}
\item The map $f$ is a weak equivalence.
\item The map $f$ induces isomorphisms, $\pi_{0}(K)\cong \pi_{0}(K')$, $\pi_{1}(K,x)\cong \pi_{1}(K',f(x))$ 
for any $x\in K_{0}$,
$H_{j}(K';\cE)\cong H_{j}(K;f^*\cE)$, for any local coefficient system $\cE$  on $K'$ and any $j$.
\end{enumerate} 
\end{lemma}

We  also use the next well-known  slight modification.

\begin{lemma}\label{lem:quillenfini}
Let $f\colon K\to K'$ be a simplicial map between Kan simplicial sets whose induced maps verify the following properties:
$\pi_{0}(K)\cong \pi_{0}(K')$, $\pi_{1}(K,x)\cong \pi_{1}(K',f(x))$ for any $x\in K_{0}$,
$H_{j}(K;f^*\cE)\cong H_{j}(K';\cE)$, for 
any local coefficient system $\cE$ on $K'$, 
 any $j\leq \ell$ and
$H_{\ell+1}(K;f^*\cE)\to H_{\ell+1}(K';\cE)$ is a surjection.
Then, the map induced by $f$ between the  homotopy groups is an isomorphism 
for $j\leq \ell$ and a surjection for $j=\ell+1$.
\end{lemma}

\begin{proof}
Let $\widetilde{K}\xrightarrow{p} K$ be the universal cover  of $K$.
The key observation in the proof given by Quillen (\cite{MR0223432}) is the degeneracy of the Serre spectral sequence,
$H_{j}(K,{p}^*\Z)\cong H_{j}(\widetilde{K};\Z)$,
for coefficients $\Z$ on $\widetilde{K}$ and the induced local coefficient system $p^*\Z$ on the homology of the fibres.
The result follows from the classic Whitehead theorem applied to the universal covers.
\end{proof}

%
\begin{proposition}\label{prop:cone}
Let $X$ be a compact filtered space 
 and $\rc X$ be the open cone, of apex $\tv$, with the conic filtration.
Let $\ov{p}$ be a   perversity on $\rc X$,  
we also denote  by $\ov{p}$ the perversity induced on $X$.
Then, for any local coefficient system $\cE$ on $\crG_{\ov{p}}X$, we have
\begin{equation}\label{equa:conegajer}
H_{j}(\crG_{\ov{p}}\rc X, \crG_{\ov{p}} X ;\cE ) =
0\quad\text{for any}\quad j\leq D\ov{p}(\tv)+1.
\end{equation}
\end{proposition}

\begin{proof}
We apply \propref{prop:isolatedgajer} to $\rc X$,  its apex $\tv$ and $\ell=D\ov{p}(\tv)+1$. 
The equality \eqref{equa:conegajer} follows from the isomorphism between
the $\ell$-skeleta, $(\crG_{\ov{p}}(\rc X))^\ell\cong (\crG_{\ov{p}}(\rc X\menos\{\tv\}))^\ell$, 
\propref{prop:homotopyxR} and the long exact 
homology sequence of the pair $(\crG_{\ov{p}}\rc X,\crG_{\ov{p}}X)$.
\end{proof}

For sake of simplicity let us take $\Z$ as coefficients. 
The determination done in \propref{prop:cone} for the homology of $\crG_{\ov{p}}\rc X$ gives a result of different spirit than the $\ov{p}$-intersection homology of a cone, $H_{*}^{\ov{p}}(\rc X)$.  
 If  $j\leq D\ov{p}(\tv)$, this last homology verifies $H_{j}^{\ov{p}}(\rc X) \cong
H^{\ov{p}}_{j}(X)$,
which is similar to \eqref{equa:conegajer}.
But we also have $H^{\ov{p}}_{j}(\rc X)=0$ if  $j> D\ov{p}(\tv)$, see \cite[Proposition 3]{CST8} or \cite[Proposition 5.2]{CST3}.
This \emph{last condition is not true for the homology of the Gajer space on a cone.}

Before  giving concrete examples, let's analyze
 the difference between the elements of $C_{*}^{\ov{p}}(Y)$ and those of $C_{*}(\crG_{\ov{p}}Y)$
 for a perverse space $(Y,\ov{p})$.
Let  $\xi=\sum_{i}r_{i}\sigma_{i}$ be a singular chain on  $Y$. For having
$\xi\in C_{*}(\crG_{\ov{p}}Y)$, the allowability condition has to be satisfied for all the faces of all $\sigma_{i}$.
In contrast, in $C_{*}^{\ov{p}}(Y)$ there is no requirement on the faces that cancel out in the boundary $\partial \xi$.
As already observed in \cite{MR1489215}, the 
\emph{homology of $\crG_{\ov{p}}Y$ is not isomorphic to the $\ov{p}$-intersection homology.}
To see that, let $Y=\rc (S^2\times S^3)$ with the perversity determined by the value 
$\ov{p}(\tv)=D\ov{p}(\tv)=2$ on the apex. 
Using the determination (\propref{prop:coneandhomotopy}) of the homotopy groups of the Gajer space of a cone, 
the only non-zero intersection homotopy group is
$\pi_{2}(\crG_{\ov{2}}Y)=\Z$. Therefore, the space $\crG_{\ov{2}}Y$ is the Eilenberg-MacLane space $K(\Z,2)$.
The homology of $\crG_{\ov{2}}Y$ does not fit with the $\ov{2}$-intersection homology groups of $Y$ which are zero in degrees strictly greater than~2.
 We continue with other examples of this feature 
 which do not need the use of $\ov{p}$-intersection homotopy groups.
\begin{example}\label{exam:gajernotintersection}
Let $L$ be  an oriented $\ell$-dimensional connected compact manifold verifying 
$ \pi_\ell (L)=0$ (as the torus for $\ell=2$), endowed with the GM-perversity $\ov{0}$.
Below, we prove: 
$$
H_\ell^{\ov 0} (\rc L;\Z) =  0 
\ne H_{\ell} (L;\Z) =H_\ell (\crG_{\ov 0}( \rc L) ;\Z).
$$
Let us see that.
The dual perversity of $\ov 0$ being $D\ov 0= \ov{\ell-1}$, we deduce from the cone formula 
for intersection homology $H_\ell^{\ov 0} (\rc L;\Z) =  0 $.
The second part of the claim holds if we prove that the natural inclusion induces an isomorphism 
$ H_\ell (L;\Z)  \to H_\ell (\crG_{\ov 0}( \rc L);\Z) $. 
This map is  an epimorphism since, at the level of chain complexes, we have 
(see \propref{prop:isolatedgajer})
$C_{\leq \ell}^{\ov 0} (L\times ]0,1[;\Z) = C_{\leq \ell} (\crG_{\ov 0} \rc L;\Z)$.

Let  $[\alpha] \in H_\ell (L;\Z) $ be the orientation class of $L$. 
We reason by the absurd and suppose   $[\alpha] =  0$ in $H_\ell (\crG_{\ov 0}( \rc L);\Z)$. 
The claim will be proven if we get a contradiction.
By hypothesis, there exists
$$\gamma = \sum_{i\in I} n_i \sigma_i  \in C_{\ell + 1} (\crG_{\ov 0}( \rc L);\Z),$$ 
with $\alpha = \partial \gamma$.
Let $\tv$ be the apex of $\rc L$.
The allowability condition of \defref{def:homotopygeom} implies that, for each simplex $\sigma_i \colon \Delta^{\ell +1} \to \rc L$, the set
$\sigma_i^{-1}(\tv)$ 
is finite and included in the interior of ${\Delta}^{\ell+1}$.
Thus, for  each  $i \in I$, the restriction $\widetilde \sigma_i \colon \partial \Delta^{\ell+1} \to \rc L$ 
of $\sigma_i$ takes value in $\rc L \menos \{ \tv\}$.
This restriction defines a homotopy class in $\pi_\ell ( \rc L \menos \{ \tv\}) = \pi_\ell (L)=0$. 
So, there exists a continuous map $\tau_i \colon \Delta^{\ell+1} \to \rc L \menos \{ \tv\}$
extending $\widetilde \sigma_i$. 
We obtain $\tau_i \in  
 C_{\ell + 1} ( \rc L \menos \{ \tv\};\Z)
$ with $\partial \tau_i = \partial \sigma_i$. Finally,
$$
[\alpha] = \left[ \sum_{i\in I} n_i  \partial \sigma_i\right] = \left[ \sum_{i\in I} n_i  \partial \tau_i\right]= 
 \left[ \partial \sum_{i\in I} n_i   \tau_i\right] = 0
$$
in  $ H_{\ell } ( \rc L \menos \{ \tv\};\Z) =  H_{\ell } ( L;\Z)
$.
We get a contradiction since $[\alpha]$ is the orientation class.
\end{example}

\medskip
Evidently, as the homology of $\crG_{\ov{p}}X$ is the homology of a space, it possesses  Mayer-Vietoris sequences. 
But the Mayer-Vietoris sequence that we are considering below is related to an open covering of the  space $X$. 
The existence of such  sequence is a consequence of a theorem of $\cU$-small chains that we first establish. 

\begin{proposition}\label{prop:Usmall}
Let $(X,\ov{p})$ be a perverse space with a locally finite stratification, $\cU$ be an open covering  of $X$
and $\cE$ be a local coefficient system on $\crG_{\ov{p}}X$. 
We denote 
$\crG^{\cU}_{\ov{p}}X$
the simplicial subset  whose simplexes are the $\ov{p}$-full simplexes included in an element of $\cU$. 
Then the following properties are verified.
\begin{enumerate}[1)]
\item There exists a chain map $\sd \colon  C_*(\crG_{\ov{p}}X;\cE) \to C_*(\crG_{\ov{p}}X;\cE) $ such that, 
for each $\ov{p}$-full simplex $\sigma$, there exists $r \in \N$ 
with $\ sd^r \sigma \in  C_*(\crG_{\ov{p}}^\cU X; \cE)$.
\item The inclusion $\iota \colon  C_*(\crG_{\ov p}^{\cU}X;\cE) \to C_*(\crG_{\ov p}X;\cE)$ induces an isomorphism in homology.
\end{enumerate}
\end{proposition}

\begin{proof}
Let $\sigma\colon \Delta^\ell \to X$ be a $\ov{p}$-full simplex.
The construction comes from a  process of subdivision applied to  $\sigma$.   
Here we have to adapt it to allowable simplexes and to the presence of local coefficients.
For  local coefficients, the adaptation is classic: if $F$ and $K$ are faces of a simplex $\Delta$, 
with the join $\sigma_{F}\ast \sigma_{K}$ defined
and $a\in \cE(\sigma_{K}(0))$, one sets
$\sigma_{F}\ast (a\,\sigma_{K})=\cE(\sigma_{F}(0),\sigma_{K}(0))^{-1}(a) \sigma_{F}\ast \sigma_{K}$.
The adaptation to the compatibility with perversities is obtained by replacing the classic barycentric subdivision
with the pseudobarycentric subdivision recalled in \remref{rem:pseusobarycenter}.
As in the classical topological case, this gives a chain map 
$\sd \colon  C_*(\crG_{\ov{p}}X;\cE) \to C_*(\crG_{\ov{p}}X;\cE) $
satisfying the first point of the statement,
see \cite[Proposition 7]{CST8} in the case of ordinary coefficients.

\smallskip
The proof of 2) uses  a morphism
$T\colon C_{*}(X;\cE)\to C_{*+1}(X;\cE)$
verifying $\id -\sd =T\circ \partial+\partial\circ T$
and sending a $\ov{p}$-allowable simplex on a $\ov{p}$-allowable chain.
Similarly to that of the previous subdivision, the construction of the map $T$ is an adaptation of the classical case, 
with an induction from the pseudobarycentric subdivision.
If $\sigma$ is a $\ov{p}$-full simplex, we obtain a chain $T(\sigma)\in C_*(\crG_{\ov p}X;\cE)$,
as in  \cite[Proposition 6]{CST8}   for the case of ordinary coefficients. 
\end{proof}

The following two properties arise from \propref{prop:Usmall}, as in the classic case (see \cite[Chapter 6]{MR957919}).

\begin{theoremb}\label{thm:MV}
Let $(X,\ov{p})$ be a  perverse space with a locally finite stratification,
$\{U,V\}$ be an open covering of $X$ and 
$\cE$ be a local coefficient system on $\crG_{\ov{p}}X$. 
Then, there exists a Mayer-Vietoris long exact sequence, 
$$
\dots \to
H_{*}(\crG_{\ov p} (U \cap V);\cE) \to H_{*}(\crG_{\ov p} U;\cE) \oplus H_{*}(\crG_{\ov p} V;\cE)  \to  
H_{*}(\crG_{\ov p} X  ;\cE) 
\to  H_{*-1}(\crG_{\ov p} (U \cap V);\cE)  \to  \dots 
$$
\end{theoremb}


\begin{theoremb}\label{thm:excision}
Let $(X,\ov{p})$ be a perverse space, $\cE$ be a local coefficient system on $\crG_{\ov{p}}X$
and $Z\subset A$ be two subspaces of $X$ 
such that $\ov{Z}$ is included in the interior of $A$. Then there is an isomorphism
$$H_{*}(\crG_{\ov{p}}(X\menos Z),\crG_{\ov{p}}(A\menos Z);\cE)\cong
H_{*}(\crG_{\ov{p}}X,\crG_{\ov{p}}A;\cE).$$
\end{theoremb}

\section{Intersection homotopy groups}
 
A pointed perverse space is a triple $(X,\ov{p},x_{0})$ where $(X,\ov{p})$ is a perverse space and 
$x_{0}\in X\menos \Sigma_{X}$ is a regular point.
We also denote $x_{0}$ the simplicial subset of $\sing\, X$ generated by $x_{0}\colon \Delta^0\to X$.

\begin{definition}{(\cite{MR1404919})}\label{def:intersectionhomotopygp}
Let $(X,\ov{p},x_{0})$ be a pointed perverse space.
The \emph{$\ov{p}$-intersection homotopy  groups}
(or perverse homotopy groups)
are the homotopy groups of the simplicial set $\crG_{\ov{p}}X$,
$$\pi^{\ov{p}}_{\ell}(X,x_{0})=\pi_{\ell}(\crG_{\ov{p}}X,x_{0}).$$
\end{definition}

As $\crG_{\ov{p}}X$ is Kan, an element of $\pi^{\ov{p}}_{\ell}(X,x_{0})$ is given by a $\ov{p}$-full simplex, 
$\sigma\colon \Delta^\ell\to X$, with $\sigma(\partial\Delta^\ell)=\{x_{0}\}$.
Two such simplexes, $\sigma_{0}$ and $\sigma_{1}$, are equivalent if there exists a $\ov{p}$-full simplex
$\Phi\colon \Delta^{\ell+1}\to X$ verifying
$\partial_{i}\Phi=x_{0}$ if $i\geq 2$ and $\partial_{i}\Phi=\sigma_{i}$ for $i=0,\,1$.
We continue with the set $\pi_{0}^{\ov{p}}(X)$ of the connected components of $\crG_{\ov p} X$.

\begin{definition}\label{def:pconnexe}
A perverse space, $(X,\ov{p})$, is said \emph{$\ov{p}$-connected} if $\pi_{0}^{\ov{p}}(X)=0$.
\end{definition}

 If $(X,\ov{p})$ is a $\ov{p}$-connected perverse space, we do not need to specify the basepoint of the homotopy groups 
 and  sometimes we will write $\pi^{\ov p}_\ell (X)$ instead of $\pi^{\ov{p}}_{\ell}(X,x_{0})$.

\begin{remark}\label{exam:pi0}
 Since each regular point of $X$ is a $\ov p$-full simplex, we have the natural map $\pi_0(X \menos \Sigma_X) \to \pi_0^{\ov p} (X)$. 
 If the perversity $\ov p$ verifies $\ov p\leq \ov t$, then the  $\ov p$-full simplexes $\sigma \colon\Delta^\ell \to X$, for $\ell=0,1$,  
 do not meet the singular part of $X$. So, $\pi_0^{\ov p}(X) = \pi_0(X\menos \Sigma_X)$.  
 On the other hand, if $\ov p \geq \ov t + \ov 2$ then the $\ov p$-allowability condition is always fulfilled by the 0 and 1 simplexes.
 So $\pi_{0}^{\ov{p}}(X)=\pi_{0}(X)$.
 \end{remark}
 
If $(X,\ov{p})$ is a perverse CS set, the set $\pi_{0}^{\ov{p}}(X)$ is related with the family of connected components of $X\menos \Sigma_X$ in a more specific way.

\begin{proposition}\label{prop:petitchemin}
Let $(X,\ov{p})$  be a perverse CS set. For any point $x\in X$, there exists a path,
$\beta\colon [0,1]\to X$, with $\beta(0)=x$, $\beta(]0,1])\subset X\menos \Sigma$.
Thus the canonical inclusion induces a surjection $\pi_{0}(X\menos \Sigma)\to \pi_{0}^{\ov{p}}(X)$.
\end{proposition}

\begin{proof}
If $x$ is regular, we choose the constant path. If not, we consider a conical chart, $\varphi\colon \R^k\times \rc L\to U$
with $\varphi(0,\tv)=x$. We join $x$  to a regular point with the path, $t \mapsto (0, [z,t])$. 
 Here, we have written $\tv = [z,0]$ where $z$ is a regular point of the link $L$, which exists  since $L\menos \Sigma_L\ne \emptyset$. 
\end{proof}
 
The first concrete computation of $\ov{p}$-intersection homotopy groups 
concerns the  basic example of the cone on a filtered space. 
The following statement is the Eckmann-Hilton dual of the intersection homology of cones
(\cite[Section 6]{GM1} or \cite[Proposition 5.2]{CST3}).  

\begin{proposition}{\cite[Example 2]{MR1404919}}\label{prop:coneandhomotopy}
Let $(X,x_{0})$ be a compact filtered space,  pointed by a regular point $x_{0}$,
 and $\rc X$ be the open cone of apex $\tv$, with the conic filtration and pointed by $y_{0}=[x_{0},1/2]$.
Let $\ov{p}$ be a   perversity on $\rc X$,  
we denote also by $\ov{p}$ the perversity induced on $X$.
Then, the $\ov{p}$-intersection homotopy groups of $\rc X$ are given by 
$$
\pi_{\ell}^{\ov{p}}(\rc X,y_{0})=\left\{
\begin{array}{lcl}
\pi_{\ell}^{\ov{p}}(X,x_{0})&\text{ if }& \ell\leq D\ov{p}(\tv),\\ 
0&\text{ if }& \ell> D\ov{p}(\tv). 
\end{array}
\right.$$
\end{proposition}

\begin{proof} 
For $\ell \leq D\ov{p}(\tv)$, the statement is a  consequence of  Propositions ~\ref{prop:isolatedgajer} 
and \ref{prop:homotopyxR}.
Consider now a  $\ov{p}$-full  simplex
$\sigma=[\sigma_{0},\sigma_{1}]\colon (\Delta^\ell,\partial \Delta^\ell)\to (\rc X,y_{0})$
with $\ell\geq D\ov{p}(\tv) +1 $. 
For $\ell>0$, the result will be established if we define a $\ov{p}$-full simplex,
$\Phi\colon \Delta^{\ell+1}\to \rc X$ such that
$\partial_{i}\Phi=y_{0}$ if $i\geq 1$ and $\partial_{0}\Phi=\sigma$.
These requirements define $\Phi_{|\partial \Delta^{\ell+1}}$. 
To extend it to the whole simplex, we consider
the barycentre $b$ of $\Delta^{\ell+1}$ and  set $\Phi(b)=\tv$.
We extend $\Phi$ linearly
to the rest of $\Delta^{\ell+1}$. 
To check the allowability condition for $\Phi$,
we distinguish according to the two possible types of singular strata.
\begin{enumerate}[(i)]
\item \emph{The  stratum is the vertex $\tv$ of the cone.}

$\bullet$ If $ \sigma^{-1}\tv \neq\emptyset$, then we have
$ \dim \Phi^{-1}\tv=\dim \rc_{b}\sigma^{-1}  \tv \leq 1+  \dim\sigma^{-1}(\tv ) \leq 1+ \ell- D\ov{p}(\tv)-2$.

$\bullet$  If $\sigma^{-1}\tv =\emptyset$, then $\Phi^{-1}\tv =b $.  %
The restriction $\ell\geq D\ov{p}(\tv)+1$ implies
$ \dim \Phi^{-1}(\tv)=0\leq1+   \ell -D\ov{p}(\tv) - 2$.%
\item \emph{The stratum is the product $S\times ]0,1[$ where $S$ is a singular stratum of $X$.}
 
 $\bullet$ If $\sigma^{-1}(S\times ]0,1[)\neq\emptyset$, then we have\\
$ \Phi^{-1}(S\times ]0,1[)=
\{tx+(1-t)b\mid 
[\sigma_{0}(x),t\sigma_{1}(x)]
\in S\times ]0,1[\}  =\sigma^{-1}(S \times ]0,1[) \times ]0,1]$.
We deduce $\dim\Phi^{-1}(S\times ]0,1[)=\dim\sigma^{-1}(S\times ]0,1[)+1$
and the same argument than above applies here.

$\bullet$ If $\sigma^{-1}(S\times ]0,1[)=\emptyset$, then
 $\Phi^{-1}(S\times ]0,1[)=\emptyset$.
\end{enumerate}
For $\ell=0$ the situation is slightly different. We have two $\ov p$-full simplexes $y_1,y_2 \colon \Delta^0 \to \rc X$ and we need to find a  
$\ov{p}$-full simplex,
$\Phi\colon \Delta^{1}\to \rc X$ such that
$\partial_1\Phi=y_1$ and $\partial_{0}\Phi=y_2$.
We leave the adaptation to the reader.
\end{proof}

This determination implies the following properties. The first one is a consequence of the 
homotopy exact long sequence of a pair.
 
\begin{corollary}\label{cor:coneandhomology}
With the hypotheses of \propref{prop:coneandhomotopy},  the relative $\ov{p}$-intersection homotopy  groups of the pair $(\rc X,X)$ verifies
$
\pi_{\ell}^{\ov{p}}(\rc X,X)=0$,  for any $\ell \leq D\ov{p}(\tv)+1$.
\end{corollary}

\begin{corollary}\label{cor:gajerpostnikov}
With the hypotheses of \propref{prop:coneandhomotopy} for a topological space $X$ with the trivial filtration,
the simplicial set $\crG_{\ov{p}}(\rc X)$ is a 
$D\ov{p}(\tv)$-Postnikov stage  of the simplicial set $\sing \,X$. 
\end{corollary}

\begin{corollary}\label{cor:gajerconemodel}
Let $\rc X$ be the cone on a simply connected, compact space $X$, filtered by its apex
 $\{\tv\}\subset \rc X$, and let $\ov{p}$ be a perversity given by $\ov{p}(\tv)$.
If $(\Lambda Z,d)$ is the  Sullivan minimal model  of $X$, 
then the cdga $(\Lambda Z^{\leq D\ov{p}(\tv)},d)$ is the minimal model of $\crG_{\ov{p}}(\rc X)$.
\end{corollary}

\begin{proof}
This  property comes from the relation between Sullivan minimal models and Postnikov towers, see \cite{MR0646078}.
\end{proof}

 In opposition with the PL situation, in a  CS set $X$,
one can encounter  links of the same point that are not homeomorphic (\cite[Example 2.3.6]{FriedmanBook}). 
In the case of intersection homology groups, this phenomenon does not matter since all these
links have the same intersection homology groups. 
We prove below that the links of points in the same stratum have isomorphic $\ov{p}$-intersection homotopy groups, 
for any perversity $\ov{p}$ on the total space.
We analyze this property again in \remref{rem:overlap}.
(Recall  that a perversity defined on $X$ induces a perversity on the conical charts thus on the links, 
see  \remref{rem:perverselink}.)

\begin{proposition}\label{prop:LinkUnique}
Let $(X,\ov p)$ be a  normal perverse CS set, 
we also denote by $\ov{p}$ the perversities induced on each link. 
Let $S$ be a singular stratum,  $x_1,x_2$ be two points of $S$ and $L_1,L_2$ be respective links of $x_{1}$ and $x_{2}$.
Then, for any $\ell\geq 1$, there is a group isomorphism,
 $$
\pi_\ell^{\ov  p} (L_1,x_{1}) \cong \pi_\ell^{\ov  p} (L_2,x_{2}).
 $$
\end{proposition}

\begin{proof}
By connectedness of $S$, we can suppose that $x_1=x_2=x$.
In the proof of the analogous result for the intersection homology  (see \cite[Lemma 5.3.13]{FriedmanBook}),
one constructs  open subsets, linked by canonical inclusions as follows,
$$V'_{1}\xrightarrow{f}V_{2}\xrightarrow{g} V_{1}\xrightarrow{h}V'_{2},
$$
with $V_{2}$ stratified homotopy equivalent to $L_{2}$, $V_{1}$ stratified homotopy equivalent to $L_{1}$
and such that $gf$ and $hg$ are stratified homotopy equivalences.
 As a stratified homotopy equivalence induces an isomorphism between intersection homotopy groups 
(\corref{cor:gajerhomotopy}),
the compositions $hg$ and $gf$ induce isomorphisms between intersection homotopy groups.
As $L_{1}$ and $L_{2}$ are connected, the map  $g$ induces an isomorphism
$\pi_\ell^{\ov  p} (L_1) \cong \pi_\ell^{\ov  p} (L_2)$.
\end{proof}

By taking the perversity $\ov{p}=\ov{\infty}$, the previous result  implies 
$\pi_{\ell}(L_{1}\menos\Sigma_{L_{1}})\cong \pi_{\ell}(L_{2}\menos\Sigma_{L_{2}})$.
With $\ov{p}=-\ov{\infty}$,
we also get $\pi_{\ell}(L_{1})\cong \pi_{\ell}(L_{2})$.

\section{Intersection fundamental group}\label{sec:piuno}

Let $(X,\ov{p})$ be a perverse  space, with $\ov{p}\leq \ov{t}$. 
Recall that any $\ov{p}$-allowable simplex, $\sigma\colon \Delta^\ell\to X$, verifies
$\dim \sigma^{-1}S\leq \ell -\codim S +\ov{p}(S)\leq \ell -2$,
for any singular stratum $S$.
Thus,  the vertices and the edges of a  $\ov{p}$-full simplex stay in the regular part.
This implies 
\begin{equation}\label{equa:pi1regular}
\text{
\emph{the surjectivity of} 
$\pi_{1}(X\menos\Sigma_{X},x_{0})\to \pi_{1}^{\ov{p}}(X,x_{0})$ for $x_{0}\in X\menos\Sigma_{X}$. 
}\end{equation}
We begin with a Van Kampen theorem if $\ov{p}\leq \ov{t}$.
The proof consists in the observation that the classic demonstration  can be done in accordance with the $\ov{p}$-allowability conditions.
For that, we use \cite[Section 1.2]{MR1867354} as guideline.

\smallskip
Let $(X,\ov{p})$ be a perverse space with $\ov{p}\leq \ov{t}$, $\{U_{0},U_{1}\}$  be an open cover of $X$. 
The open subsets $U_{0}$, $U_{1}$, $U_{0}\cap U_{1}$ are supposed to be path-connected and we endow them with the induced filtration and the induced perversity.
We denote by
$j_{k}\colon \pi_{1}^{\ov{p}}(U_{0}\cap U_{1},x_{0})\to \pi_{1}^{\ov{p}}(U_{k},x_{0})$ 
and 
$\iota_{k}\colon \pi_{1}^{\ov{p}}(U_{k},x_{0})\to  \pi_{1}^{\ov{p}}(X,x_{0})$
the homomorphisms induced by the canonical inclusions for $k=0,\,1$.
The homomorphisms $\iota_{k}$ extend to a homomorphism defined on the free product of groups,
$$\Phi\colon  \pi_{1}^{\ov{p}}(U_{0},x_{0})\ast \pi_{1}^{\ov{p}}(U_{1},x_{0})\to  \pi_{1}^{\ov{p}}(X,x_{0}).$$
The Van Kampen theorem consists of a determination of the kernel and the image of  $\Phi$.

\begin{theoremb}\label{thm:VK}
The map $\Phi$ is surjective and its kernel is the normal subgroup generated by all elements of the form
$j_{0}(\omega)\ast j_{1}(\omega)^{-1}$.
\end{theoremb}

\begin{proof}
Let $\alpha$ be a $\ov{p}$-allowable loop in $X$. We already noted that the inclusion map induces a
surjective homomorphism
$\pi_{1}(X\menos\Sigma_{X},x_{0})\to \pi_{1}^{\ov{p}}(X,x_{0})$.
Thus, the $\ov{p}$-homotopy class of $\alpha$ is the image of the homotopy class of a loop 
$\beta$ in  $X\menos\Sigma_{X}$.
From the compactness of $[0,1]$, we obtain a decomposition 
$[\alpha]=\Phi([\gamma_{0}]\ast\dots\ast [\gamma_{\ell}])$
with the support of each $\gamma_{j}$ in one of the two subsets $U_{0}$, $U_{1}$. 
By using the path-connectedness of $U_{0}$, $U_{1}$, $U_{0}\cap U_{1}$, we can suppose that
each $\gamma_{i}$ is a loop based in $x_{0}$. This first step gives the surjectivity of $\Phi$.
 
\smallskip
The critical point of the proof is the determination of the kernel of $\Phi$. 
We consider two factorizations as above,
$\Phi([\gamma_{0}]\ast\dots\ast[\gamma_{\ell}])=\Phi([\gamma'_{0}]\ast\dots\ast [\gamma'_{k}])$.
Thus, the two loops $\gamma_{0}\dots\gamma_{\ell}$ and $\gamma'_{0}\dots\gamma'_{k}$ are $\ov{p}$-homotopic
and there exists a full $\ov{p}$-homotopy
$F\colon [0,1]\times [0,1]\to X$
between them.
Unlike the situation of paths, the support of the homotopy is not necessarily in $X\menos \Sigma_{X}$; i.e.
the subset $F^{-1}(\Sigma_{X})$ is not necessarily empty.
However the $\ov{p}$-allowability condition gives us a control on it and we know that 
$F^{-1}(\Sigma_{X})$  consists in a finite number of points located in the interior of the square. 
Without loss of generality, by using an appropriate subdivision,  we can suppose that $F^{-1}(\Sigma_{X})$
is reduced to one point $v$ in the interior of $[0,1]\times [0,1]$.
As in the surjective part of the proof, we use the compactness of $[0,1]\times [0,1]$ to get a subdivision of it
so that each small square  lies either in $U_{0}$ or in $U_{1}$. 
Here, we can arrange the subdivision in such a way that the point $v$ does not belong to any of the boundaries
of the small squares. This implies that each small square is  a full $\ov{p}$-homotopy in $U_{0}$ or $U_{1}$.
The rest of the proof consists to the use of the path-connectedness of $U_{0}$, $U_{1}$, $U_{0}\cap U_{1}$
to transform each small square in a $\ov{p}$-homotopy between loops based in $x_{0}$.
This is done with the introduction of paths from $x_{0}$ to the corners of the small squares. 
As this can be done in $X\menos \Sigma_{X}$, 
each of the small squares remains a full $\ov{p}$-homotopy.
This gives the conclusion, as in the classic topological situation.
\end{proof}

\begin{corollary}\label{cor:thommather}
Let $X$ be a connected normal Thom-Mather space (\cite{MR2958928}) with a finite number of strata.
Then, for any $j\geq 0$, there is an isomorphism
$$\pi_{j}^{\ov{t}}(X)\cong \pi_{j}(X).$$
\end{corollary}

\begin{proof}
From \remref{rem:conexenormal}, we deduce the connectivity of $X\menos \Sigma$ and we can omit the reference to the basepoint.
Let $S$ be one of the minimal strata. There is a locally trivial fibration
$\rc L\to E\to S$, playing the role of a tubular neighborhood.
From \cite[Theorem 2.3]{MR1404919}, we deduce a long exact sequence
$\ldots\to \pi_{1}^{\ov{t}}(\rc L)\to \pi_{1}^{\ov{t}}(E)\to \pi_{1}^{\ov{t}}(S)\to\dots$.
The space $S$ has only one stratum so $\pi_{1}^{\ov{t}}(S)=\pi_{1}(S)$.
 From \propref{prop:coneandhomotopy}, we know $\pi_{1}^{\ov{t}}(\rc L)=0$.
Thus we deduce $\pi_{1}^{\ov{t}}(E)=\pi_{1}(S)$.
Now  we do an induction on the depth and the number of minimal strata. The inductive step is reduced to an open cover of $X$ given by
$E$ and $X\menos S$. 
The hypotheses of connectivity are satisfied: 
for $E$, it is a consequence of the connectivity of $S$ and $\rc L$, 
for $X\menos S$ it comes from $X\menos \Sigma$ connected, $X\menos \Sigma\subset X\menos S$ and $\ov{X\menos \Sigma}=X$.
Finally,  $E\menos S$ is the total space of a locally trivial fiber bundle, of basis $S$ and fiber $L\times ]0,1[$, thus $E\menos S$  is connected since $L$ and $S$ are so.
The intersection fundamental group of  $E$ is computed as above. On the two open subsets, $X\menos S$, $(X\menos S)\cap E=E\menos S$, we have the induction hypothesis.
Thus, from \thmref{thm:VK},  we deduce  $\pi_{1}^{\ov{t}}(X)\cong \pi_{1}(X)$.

With \lemref{lem:quillen}, we can now replace homotopy groups by homology groups in a local system of coefficients, $\cE$.
We use a theorem of King (\cite[Section 3]{MR800845}) detailed in \cite[Theorem 5.1]{CST3},
applied to the natural transformation
$\psi_{X}\colon H_{*}(\crG_{\ov{t}}X;\cE)\to H_{*}(X;\cE)$. 
This theorem says that such transformation is an isomorphism if
 its restrictions to the open subsets of CS sets verify the following properties.
\begin{enumerate}[i)]
\item There are Mayer-Vietoris sequences and $\psi$ induces a commutative diagram between them. 
\item If $(U_{\alpha})$ is an increasing sequence of open subsets and $\psi_{U_{\alpha}}$ is an isomorphism 
for each $\alpha$, then $\psi_{\cup_{\alpha} U_{\alpha}}$ is an isomorphism.
\item If $L$ is a compact filtered space such that $X$ has an open subset $U$ stratified homeomorphic to
$\R^i\times \rc L$ and if $\psi_{U'}$ is an isomorphism for $U'=\R^i\times (\rc L\menos \{\tv\})$ (where $\tv$ is the apex)
then so is $\psi_{U}$.
\item If $U$ is an open subset of $X$ contained within a single stratum and homeomorphic to 
an Euclidean space, then $\psi_{U}$ is an isomorphism.
\end{enumerate}
As $H_{*}(\crG_{\ov{t}}X;\cE)$ admits Mayer-Vietoris sequences (\thmref{thm:MV}), 
the point iii) is the only one which deserves attention.
Looking to its conclusion, we notice from Propositions~\ref{prop:homotopyxR} and \ref{prop:coneandhomotopy},
that
$\pi_{*}(\crG_{\ov{t}}(\R\times \rc L))\cong \pi_{*}(\crG_{\ov{t}}\rc L)$
is trivial except in degree 0. The hypothesis of normality gives 
$\pi_{0}(\crG_{\ov{t}}(\R\times \rc L))\cong \pi_{0}(L)\cong 0$. 
As $\pi_{i}(\R\times \rc L)=0$ for any $i$, the result follows from the Whitehead theorem.
%
\end{proof}

We are interested in the topological invariance of the intersection fundamental group of CS sets for GM-perversities.
The following example shows that this property is not satisfied in general.

\begin{example}\label{ex:poincaresphere}
 The suspension $\Sigma M$ of a topological space $M$ being the main ingredient of the example, let us begin with
 two basic properties.\\
 1) If $\Sigma M$ is homeomorphic to a sphere $S^n$, then $M$ has the homotopy type of
 $S^{n-1}$. (Just remove the two vertices of $\Sigma M$ to get a homeomorphism
 $M\times ]-1,1[\cong S^{n-1}\times ]-1,1[$.)\\
 2) If $\Sigma M$ is homeomorphic to a sphere $S^n$ and $M$ is a manifold, 
 from 1) and the theorems of 
 Smale-Freedman-Donaldson-Perelman (also called ``Poincar\'e conjecture''), then $M$ is a sphere.
 
Let $P$ be the Poincar\'e sphere, a manifold of dimension three with the homology of the 3-sphere and such that 
$\pi_{1}(P,v)\neq 1$, with $v\in P$.
From the previous recalls, we deduce that the suspension $\Sigma P$ is not homeomorphic to $S^4$
and has  no structure of manifold,  but
has the homotopy type of $S^4$.
Let  $\tv_{1},\tv_{2}$ be  the two apexes of  $\Sigma P$, constituting the singular set of $\Sigma P$.
We choose a GM-perversity $\ov{p}$ with $D\ov{p}(4)=1 $. 
By definition,  two $\ov{p}$-allowable simplexes,
$\xi\colon (\Delta^1,\partial \Delta^1)\to (\Sigma P,v)$ and $\eta\colon \Delta^2\to \Sigma P$,
  verify
$\dim\xi^{-1}(\tv_{i})\leq \dim \Delta-2-D\ov{p}(4)=-2$
and $\dim\eta^{-1}(\tv_{i})\leq -1$. Thus, we have $\pi_{1}^{\ov{p}}(\Sigma P,v)=\pi_{1}(P,v)$.

We now consider the double suspension  of $P$ and denote by $\tw_{1},\tw_{2}$ the two new apexes.
We filter $\Sigma^2 P$ by 
$\{\tw_{1},\tw_{2}\}\subset \Sigma\{\tv_{1}, \tv_{2}\}\subset \Sigma^2 P$.
We choose a GM-perversity such that  $D\ov{p}(4)= D\ov{p}(5)=1$. 
A similar computation gives $\pi_{1}^{\ov{p}}(\Sigma^2 P,v)=\pi_{1}^{\ov{p}}(\Sigma P,v)=\pi_{1}(P,v)\neq 1$.
As $\Sigma^2 P$ is homeomorphic to the sphere $S^5$, this example shows that the intersection fundamental group is not a
topological invariant in general.
This situation does not appear in the PL case since the homeomorphism between $\Sigma^2 P$ and $S^5$
is not a PL map. 
In fact, in \cite[Theorem 2.7]{MR1404919} Gajer shows that the $\ov{p}$-intersection homotopy groups of a polyhedron is a PL invariant.
\end{example}

\begin{remark}\label{rem:overlap}
The overlapping of \propref{prop:LinkUnique} and \exemref{ex:poincaresphere} calls out. 
In the first reference, we prove that links of a point in a CS set have isomorphic intersection homotopy groups 
for any perversity. In the second one, we use the example of two CS set structures  on a particular suspension space,
with an ad'hoc perversity, to highlight 
the  failure of a general topological invariance of intersection homotopy groups.

 \smallskip
 Let us first make a few reminders about the properties of links in CS sets. 
 We  consider two CS set structures on the same topological space $X$ and a point $x\in X$.
 Let $N=(\rc L,\tv)$ and $N'=(\rc L',\tv')$ be two conical charts of $x$ in the respective CS set structures.
 From a result of Stallings on conical neighborhoods (see \cite[Corollary 2.10.2]{FriedmanBook}),
 there exists a relative homeomorphism between $N$ and $N'$.
  (This does not imply the existence of a homeomorphism between the links!) 

 \smallskip
 In the case of a \emph{fixed} structure of CS set on a topological space $X$, the invariance of the intersection
 homotopy groups of the links, established in  \propref{prop:LinkUnique},
 follows from the existence of stratified homotopy equivalences between charts and links. 
 
 \smallskip
 Let us consider two homeomorphic spaces with CS set structures, and a GM-perversity $\ov{p}$.
\emph{The topological invariance} means that the corresponding $\ov{p}$-intersection homotopy groups are isomorphic. 
That is, we now have \emph{two} different CS set structures on the ``same'' topological space and we ask wether
the intersection homotopy groups are isomorphic.
Our test for this topological invariance is the double suspension on a Poincar\'e homological 3-sphere, $P$,
which is known to be homeomorphic to $S^5$. 
 We endow  $\Sigma^2 P$ with the filtration
 $\{\tw_{1},\tw_{2}\}\subset \Sigma\{\tv_{1}, \tv_{2}\}\subset \Sigma^2 P$
 of \exemref{ex:poincaresphere}.
 The singular set is the circle $\gamma= \Sigma\{\tv_{1}, \tv_{2}\}$.
Let $\varphi\colon \Sigma^2 P\to S^5$ be the Edwards homeomorphism. 
 We get \emph{two different structures} of CS set on the topological space $S^5$:
a ``trivial'' one comes from the manifold structure of $S^5$ and a second one is the image by $\varphi$
of the previous chosen filtration on $\Sigma^2 P$.
As we only consider  these two structures in this remark, we call ``non-trivial'' the second one.
As said above,  conical charts in $\tw_{1}$ (or $\tw_{2}$)
are homeomorphic, so
$\R^0\times \rc S^4\cong \R^0 \times \rc \Sigma P$.
We emphasize that this is only a topological result without mention of filtrations. 
This is a crucial point in the analysis of links made below.

\smallskip
Set $\tu=\varphi(\tw_{1})\in S^5$. From  the manifold structure, the point $\tu$ admits an open neighborhood
homeomorphic to $\rc S^4$. We endow $S^4$ with the filtration induced by the non-trivial filtration of $S^5$.
At the request of a referee, we study whether the sphere $S^4$,
 equipped with this filtration, can be a link of $\tu$ in the non-trivial CS set structure of $S^5$. 
We prove:

\noindent\centerline{
{\tt{Claim:}} \emph{For the non-trivial filtration of $S^5$, $S^4$ is not a link of $\tu$.}
}

\smallskip
Let us suppose that $S^4$ is such a link; i.e.,
$S^4$ and $\Sigma P$ are two links of $\tu$.
The claim will be inferred from the existence of a contradiction.
We choose a perversity $\ov{p}$ such that $D\ov{p}(4)=D\ov{p}(5)=1$.
Having only one CS set structure on $S^5$ (the non-trivial one), we can apply
\propref{prop:LinkUnique} and get $\pi_{1}^{\ov{p}}(S^4)\cong \pi_{1}^{\ov{p}}(\Sigma P)$.
From a computation made in \exemref{ex:poincaresphere}, we deduce 
\begin{equation}\label{equa:pipcontradiction}
\pi_{1}^{\ov{p}}(S^4)\cong \pi_{1}^{\ov{p}}(\Sigma P)\neq 1.
\end{equation}
To obtain a contradiction, we directly determine this fundamental intersection group.
Let us first note that $\varphi(\gamma)\cap S^4$ is a set of isolated points, possibly empty.
By definition,  two $\ov{p}$-allowable simplexes,
$\xi\colon \Delta^1 \to S^4$ and $\eta\colon \Delta^2\to S^4$,
  verify
$\dim\xi^{-1}\varphi(\gamma)\leq \dim \Delta-2-D\ov{p}(4)=-2$
and $\dim\eta^{-1}\varphi(\gamma)\leq -1$.
Therefore $\pi_{1}^{\ov{p}}(S^4)\cong \pi_{1}(S^4\menos (\varphi(\gamma)\cap S^4))$.
 From \eqref{equa:filteredlink} in the definition 
 of a CS set, we deduce that, 
for each $x\in \varphi(\gamma)\cap S^4$, the segment $[x,\tu]$ is included in $\varphi(\gamma)$.
The map $\varphi$ being an embedding, the intersection
$\varphi(\gamma)\cap S^4$
is a set of at most two points. In any of the three possible cases, we have 
$\pi_{1}(S^4\menos\varphi(\gamma)\cap S^4)=1$ and a contradiction with \eqref{equa:pipcontradiction}.

\smallskip
In conclusion, the obstruction to $S^4$ being a link is the non-triviality of the intersection homotopy group 
$\pi_{1}^{\ov{p}}(\Sigma P)$.
Let us also point out that the topological embedding $\gamma\colon S^1\to S^5$ is  wild and far 
from a smooth one, as shown for instance in \cite{MR4030354}.
 \end{remark}


In \exemref{ex:poincaresphere}, some singular points of $\Sigma^2 P$  become regular in $S^5$. 
The next result proves that is the only obstruction to having a topological invariant.
Recall that a \emph{coarsening} of a filtered space $X$ is a filtered space $Y$ 
on the same topological space, such that each stratum of $Y$  
is a union of strata of $X$. The identity induces a stratified map, $\nu\colon X\to Y$.

\begin{definition}\label{def:source}
Let $\nu\colon X\to Y$ be a coarsening of a filtered space $X$. An \emph{exceptional stratum} is a singular stratum of $X$ which is included in a regular stratum of $Y$.
 A \emph{source stratum of a stratum $T$} of $Y$ is a stratum $S$ of $X$ included in  $T$ and of the same dimension.  
\end{definition}

\begin{theoremb}\label{thm:thepione}
Let $\nu\colon X\to Y$ be a coarsening of CS sets and
$\ov{p}$ be a codimensional perversity  such that $\ov{p}\leq \ov{t}$. 
If there is no exceptional stratum, then
the map $\nu$ induces isomorphisms, 
$\pi_{0}^{\ov{p}}(X)\cong \pi_{0}^{\ov{p}}(Y)$ and
$\pi_{1}^{\ov{p}}(X,x_{0})\cong \pi_{1}^{\ov{p}}(Y,x_{0})$,
for any regular point $x_{0}$.
\end{theoremb}

Any GM-perversity satisfies the required hypothesis on $\ov{p}$.
The condition imposed on the strata means that the two singular sets are identical and we denote it by $\Sigma=\Sigma_{X}=\Sigma_{Y}$.

\begin{proof}
The first assertion comes from \remref{exam:pi0}
and the surjectivity of $\nu_{*}\colon \pi_{1}^{\ov{p}}(X,x_{0})\to \pi_{1}^{\ov{p}}(Y,x_{0})$ from \eqref{equa:pi1regular}.
For the injectivity of $\nu_{*}$, we consider a $\ov{p}$-allowable loop $\xi\colon [0,1]\to X$ and suppose that there exists a 
$\ov{p}$-allowable simplex $\eta\colon \Delta^2\to Y$, 
whose boundary is $\xi$ on one edge and the constant 
loop on $x_{0}$ on the other two. 
We have to construct a  $\ov{p}$-allowable simplex
$\eta'\colon \Delta^2\to X$ such that $\partial \eta=\partial \eta'$.
From the $\ov{p}$-allowability condition on $\eta$, we know that the subset
$$\eta^{-1}\Sigma=
\cup\left\{ \eta^{-1}T\mid T \;\text{is a singular stratum of }\; Y \;\text{with} \;D\ov{p}(\codim T)=0\right\}$$
is finite. If $\eta^{-1}\Sigma=\emptyset$, we choose $\eta'=\eta$.
Otherwise, by subdivision, we can suppose that $\eta^{-1}\Sigma=\{p\}\in \inte\Delta^2$
and that the image of $\eta$ is included in a conical chart $U$. 
Let us denote by $\varphi\colon \R^m\times \rc L\to U$ this conical chart with $\varphi(0,\tv)=\eta(p)$,
where $\tv$ is the apex of the cone.
The point $\eta(p)$ belongs to a singular stratum $S$ of $X$ and to the singular stratum $S^\nu$ of $Y$.
By the previous  choices, we have $D\ov{p}(\codim S^\nu)=0$.
The family of source strata of $S^\nu$ being an open dense subset of $S^\nu$, we can find a source
stratum $Q$ of $S^\nu$ and a point $x=\varphi(u_{0},\tv)\in U$ as close as we want to $\eta(p)$.
The map $\varphi^{-1}\circ \eta\colon \partial \Delta^2\to \R^m\times \rc L$ can be written
$$\varphi^{-1}(\eta(y))=(f_{1}(y),[f_{2(}y),f_{3}(y)]).$$
Its image is included in $L\menos\Sigma_{L}$ and the map $f_{3}$ does not vanish.
Let us denote by $\flat$ the barycentre of $\Delta^2$.
We define $\eta'\colon \Delta^2\to U$ by
$$\eta'(ty+(1-t)\flat)=\varphi((1-t)u_{0}+tf_{1}(x),[f_{2}(x),tf_{3}(x)]).$$
By construction, we have
${\eta'}^{-1}\Sigma={\eta'}^{-1}Q=\{\flat\}$
and for $t=1$ we get $\partial\eta=\partial\eta'$.
As $Q$ is a stratum of $S^\nu$ it has the same codimension, and we have
$$\dim{\eta'}^{-1}\Sigma=\dim {\eta'}^{-1} Q=\dim \{\flat\}=
2-\codim Q-2=\dim \eta-\codim Q -2.
$$
This gives the $\ov{p}$-allowability of $\eta$ in the CS set $X$, as required.
\end{proof}


Among the coarsenings of a CS set $X$, there is the  intrinsic coarsest CS set, $\nu\colon X\to X^*$,
whose properties are recalled below. 
\thmref{thm:thepione} applied to this coarsest stratification  implies that the intersection
fundamental group for a codimensional perversity is a topological invariant if  no singular stratum of $X$ becomes regular in $X^*$. 
\exemref{ex:poincaresphere} shows the necessity of this restriction.

\medskip
\begin{recall}\label{rec:intrinsic}{(\bf{Intrinsic CS set})}
Let $X$ be a CS set. For the construction of a new structure of CS set on the same topological space in \cite{MR800845}, 
King utilizes an equivalence relation he credits Dennis Sullivan with: two points $x_{0}$ and $x_{1}$ are equivalent if there exist neighborhoods $U_{i}$ of $x_{i}$ with a homeomorphism $(U_{0},x_{0})\cong (U_{1},x_{1})$. 
The equivalence classes are union of strata and the CS set $X$ has a new filtration by chosing $X_{j}^*$ as the
union of equivalence classes which only contains components of strata of dimension less than or equal to $j$.
This defines a CS set denoted by $X^*$ and called \emph{the intrinsic CS set associated to  $X$.}
The identity map  is a stratified map   denoted by $\nu\colon X\to X^*$.
By construction, a stratum $T$ of $X^*$ is a locally finite union of strata of $X$. 
The set of source strata (\defref{def:source}) of $T$ is a dense open subset of $T$.

The CS set $X^*$ is a coarsening of $X$. We can observe that a singular stratum of $X$ can be included in a singular stratum or in a regular stratum of $X^*$ but a regular stratum of $X$ stays regular in $X^*$. 
Thus  the singular set $\Sigma_{X^*}=X\menos X^*_{n-1}$ is a closed subspace of $\Sigma_{X}$.
The CS sets $X$ and $X^*$ have the same underlying topological space; if there is no ambiguity, we denote it by $X$.
If $U$ is an open subset of $X$, we also denote by $U$ the CS set with the structure induced by $X$
and by $U^*$ the CS set with the structure induced by $X^*$. 
Let us recall from \cite{MR800845} how are the  structures of $X$ and $X^*$ in the neighborhood of a point.
In the following diagram,
\begin{equation}\label{equa:XandX}
\xymatrix{
\R^k\times \rc W\ar[r]^-{\varphi}\ar[d]_-{h}&
V\ar[d]^-{f}\\
\R^m\times\rc L  \ar[r]^-{\psi}&
V^*,
}\end{equation}
the map $\varphi$ is a conical chart for $X$ and $\psi$ a conical chart for $X^*$.
The maps $\varphi$ and $\psi$ are stratified homeomorphisms but $f$ and $h$ are only
homeomorphisms and stratified maps.
The space $W$ is a link of $X$ and  $L$ is a link of $X^*$. We have  $m\geq k$. 
(In the case of a regular point, the link is the empty set.)
Without loss of generality, we can suppose $h(0,\tw)=(0,\tv)$, where $\tv$ and $\tw$ are the apexes of the cones.
We denote $s=\dim W$, $t=\dim L$ and deduce $s\geq t$ from $s+k=m+t$. 
The map $h$ also verifies $h(\R^k\times \{\tw\})=B\times \{\tv\}$, with $B$ closed, and
$h^{-1}(\R^m\times \{\tv\})=\R^k\times \rc A$. 
By writing $\R^m\cong \R^k\times \rc A$ as
$\rc S^{m-1}\cong \rc S^{k-1}\times \rc A\cong \rc(S^{k-1}*A)$, we see that
$A$ is a   homology  sphere of dimension $(m-k-1)$. 
In particular, $W$ itself is an $(m-k-1)$-dimensional homological sphere when $L=\emptyset$.
\end{recall}

 \begin{definition}\label{def:clivage}
 Let $\ov{p}$ be a GM-perversity. The \emph{cleaving point} of $\ov{p}$ is the number
 $\ell_{\ov p} \in \overline \N$ defined by
 $$
\ell_{\ov p} = \sup \,\{ k \in \overline \N \mid \ov p(k) = \ov t(k)\} =  \sup \,\{ k \in  \overline \N \mid D\ov p(k) = 0\}.
$$
\end{definition}

In the general case of a filtered space, we show that the cleaving point reduces the determination
of the $\ov{p}$-intersection fundamental groups to the case  of the top perversity.

\begin{proposition}\label{prop:letopjevousdis}
Let $\ov p$ be a GM-perversity, $X$ be an $n$-dimensional  filtered space without stratum of codimension~1 
and $x_{0}$ be a regular point of $X$. 
Then, we have
$$
\pi_1^{\ov p}(X, x_{0}) =  \pi_1^{\ov t} (X \menos X_{n-\ell_{\ov p}-1},x_{0}).
$$
\end{proposition}

\begin{proof} 
In the proof of \thmref{thm:thepione}, we have already noticed that the 1-skeleton of $\crG_{\ov{p}}X$ is isomorphic 
to the 1-skeleton of $\sing (X\menos \Sigma_{X})$ and that a 2-simplex, $\eta\colon \Delta^2\to X$ is $\ov{p}$-full
if, and only if,
$ \dim \eta^{-1}S \leq 2 - D\ov p (\codim S) -2 = - D\ov p (\codim S)$ for any singular stratum $S$.
Since $D\ov{p}(\codim S)\geq 0$, we distinguish two cases.

$\bullet$ If $D\ov p(\codim S)\geq 1$, which corresponds to $\codim S \geq \ell_{\ov p} +1$ 
or $\dim S \leq n -\ell_{\ov p}-1$, then we have
$ \dim \eta^{-1}S \leq  - D\ov p (\codim S) \leq -1$. This implies
$\eta^{-1}S=\emptyset$, which is equivalent to $\eta \colon \Delta^2 \to X \menos X_{n -\ell_{\ov p} -1}$,
with the restriction on $S$.

$\bullet$ If $D\ov p(\codim S)= 0 $, which corresponds to   $\codim S \leq \ell_{\ov p}$  
or $\dim S \geq n -\ell_{\ov p}$ then we have
$
\dim \eta^{-1}S \leq - D\ov p(\codim S)=0 $
 which is equivalent to $ \dim \eta^{-1}S    \leq 2 - D\ov t(\codim S) -2$,
with the restriction on $S$.

\smallskip
In short, the $\ov{p}$-full 2-simplexes are exactly the simplexes  $\eta \colon \Delta^2 \to X \menos X_{n -\ell_{\ov p} -1}$ 
such that $\dim \eta^{-1}S \leq 2 - D\ov t(\codim S) -2$ for any stratum $S$ of  $X \menos X_{n -\ell_{\ov p} -1}$.
We have proven 
$$ 
\crG_{\ov p}^2(X) = \crG_{\ov t}^2(X\menos X_{n -\ell_{\ov p} -1})
$$
and, therefore,
$\pi_1^{\ov p}(X,x_{0}) = \pi_1^{\ov t}(X\menos X_{n -\ell_{\ov p} -1},x_{0})$.
\end{proof}

Under the hypothesis of \corref{cor:thommather}, we can reduce  
the determination of  $\ov{p}$-intersection fundamental groups to that of classic fundamental groups with
$\pi_{1}^{\ov{p}}(X,x_{0})=\pi_{1}(X \menos X_{n-\ell_{\ov p}-1},x_{0})$.

\section{Topological invariance of intersection homotopy groups}\label{sec:invariance}

In this section, we study the topological invariance of perverse homotopy groups and prove it 
when the regular parts of $X$ and of its intrinsic filtration coincide.
Let us notice that this is a condition on the topological space $X$ itself: 
it means that there is no regular point equivalent to a singular point for the relation of King
(\cite{MR800845}) quoted in Recall~\ref{rec:intrinsic}.
We have already proven that the perverse fundamental group is a topological invariant under the same hypothesis (\thmref{thm:thepione}) and
that there are counterexamples in the general case (\exemref{ex:poincaresphere}).

\begin{theoremb}\label{thm:homotopyinvariance}
Let $\ov{p}$ be a GM-perversity  and $X$ be a CS set without  stratum of codimension 1. 
We denote by  $X^*$  the associated CS set  with the intrinsic filtration.
If there is no exceptional stratum, then
the map $\nu\colon X\to X^*$ induces an isomorphism
$\pi_{j}^{\ov{p}}(X,x_{0})\cong \pi_{j}^{\ov{p}}(X^*,x_{0})$,
for any $j$ and any regular point $x_{0}$.
\end{theoremb}

The proof begins with a local version in a conical chart.

\begin{proposition}\label{prop:invarianceopensubsetnew}
Let $\ov{p}$ be a GM-perversity, $X$ be a  CS set 
without stratum of codimension 1.
We denote by $X^*$  the associated CS set  with the intrinsic filtration.
Let $S$ be a singular stratum of $X$ and $V$ be a conical chart of $x\in S$. 
We suppose that the stratum $S$ is included in a singular stratum of $X^*$
and that there are isomorphisms,
$\pi_{0}^{\ov{p}}\nu\colon \pi_{0}^{\ov{p}}V\cong  \pi_{0}^{\ov{p}}V^*$
and
$\pi_{1}^{\ov{p}}\nu\colon \pi_{1}^{\ov{p}}(V,x_{0}){\cong}  \pi_{1}^{\ov{p}}(V^*,x_{0})$, 
for any regular point $x_{0}$.
If the map
\begin{equation}\label{equa:hypinvariancenew}
\pi_{j}^{\ov{p}}\nu\colon \pi_{j}^{\ov{p}}(V\menos S,x_{0})\xrightarrow{\;\;\cong\;\;} \pi_{j}^{\ov{p}}(V^*\menos S),x_{0}),
\end{equation}
is an isomorphism for any $j$, then the next map is also an isomorphism,
\begin{equation}\label{equa:conclusioninvariancenew}
\pi_{j}^{\ov{p}}\nu\colon \pi_{j}^{\ov{p}}(V,x_{0})\xrightarrow{\;\;\cong\;\;}  \pi_{j}^{\ov{p}}(V^*,x_{0}).
\end{equation}
\end{proposition}

\begin{proof}
We can suppose  that $V\cong \R^k\times \rc W$ is a conical chart, with $\tw$ the apex of the cone $\rc W$. 
Thus, there is a homeomorphism $S\cap V\cong \R^k\times \{\tw\}$.
We take  the description  given in Recall~\ref{rec:intrinsic}.
We have $V^*\cong\R^m\times \rc L$ with $\tv$ the apex of the cone $\rc L$ and $m\geq k$. 
If we denote by $B\times \{\tv\}$ the image of $\R^k\times \{\tw\}$ under the previous homeomorphism, we also have $(V^*\menos S)\cong\R^m\times \rc L\menos (B\times \{\tv\})$.
We set $s=\dim W$ and $t=\dim L$; they verify $s\geq t$ since $s+k=m+t$.
The hypothesis on the stratum $S$  means $s\geq 1$ and $t\geq 1$.
For the rest of the proof, we shorten the notation as follows,
$$
\begin{array}{lll}
 V^*\cong P=\R^m\times \rc L, &  V^*\menos S\cong Q=(\R^m\times \rc L)\menos B\times \{\tv\}, &   R=(\R^m\times \rc L)\menos (\R^m\times\{\tv\}),\\
  V\cong P'=\R^k\times \rc W, &  V\menos S\cong Q'=(\R^k\times \rc W)\menos (\R^k\times \{\tw\}). &
\end{array}
$$
By using Lemmas~\ref{lem:quillen}, \ref{lem:quillenfini} and the hypotheses on the perverse fundamental groups,
the existence of the isomorphisms \eqref{equa:hypinvariancenew} and \eqref{equa:conclusioninvariancenew}
between perverse homotopy groups is equivalent to the existence of isomorphisms in homology with local coefficients.
Let us consider a local coefficient system, $\cE$, on $\crG_{\ov{p}}(P)$. 
We abuse notation and still denote by $\cE$ the local coefficient systems on 
$\crG_{\ov{p}}(P')$, $\crG_{\ov{p}}(Q)$, $\crG_{\ov{p}}(Q')$, $\crG_{\ov{p}}(R)$,
obtained by pullback along $\nu$ or by induction on a simplicial subset.
Under theses conditions, we know there exist homology exact sequences of pairs (or triples) and excisions (\thmref{thm:excision}).
If $D\ov{p}(s+1)=0$, as $\crG_{\ov{p}}(V)$ and $\crG_{\ov{p}}(V^*)$ are simply connected, we  choose 
the constant system with coefficients in $\Z$.
We do not mention explicitly the coefficient $\cE$ or $\Z$ in the rest of the proof.
We proceed in two steps.

\smallskip\noindent
$\bullet$  Let $i\leq D\ov{p}(s+1)$. 
We set $Z=B\times (\rc L\menos\{\tv\})$.
As $Z\subset R\subset Q$ and the closure of $Z$ being included in the interior  of $Q$, we can use an excision (\thmref{thm:excision})
and get
$$H_{i}(\crG_{\ov{p}}Q,\crG_{\ov{p}}R)\cong H_{i}(\crG_{\ov{p}}(Q\menos Z),\crG_{\ov{p}}(R\menos Z)).$$
Taking their image by the restriction (see \eqref{equa:XandX}) of $h^{-1}$  gives stratified homeomorphisms
$Q\menos Z\cong ((\R^k\times \rc A)\menos (\R^k\times \{\tw\}))\times \rc L$
and
$R\menos Z\cong ((\R^k\times \rc A)\menos (\R^k\times \{\tw\}))\times (\rc L\menos\{\tv\})$,
where $\tw$ is the apex of $\rc A$.
By using \propref{prop:homotopyxR} and the fact that $A$ is a homology sphere of dimension $m-1-k$, we deduce
for \emph{any} $i$,
\begin{eqnarray}
H_{i}(\crG_{\ov{p}}Q,\crG_{\ov{p}}R)
&\cong &
H_{i}(\crG_{\ov{p}}(A\times \rc L),\crG_{\ov{p}}(A\times (\rc L\menos \{\tv\})))\nonumber \\
&\cong&
H_{i}(\sing\, A\times \crG_{\ov{p}}(\rc L),\sing \,A\times \crG_{\ov{p}}(\rc L\menos \{\tv\}))\nonumber \\
&\cong&
H_{i}(\crG_{\ov{p}}(\rc L),\crG_{\ov{p}}(\rc L\menos \{\tv\}))\oplus
H_{i-m+1+k} (\crG_{\ov{p}}(\rc L),\crG_{\ov{p}}(\rc L\menos \{\tv\}))\nonumber\\
&\cong&
H_{i}(\crG_{\ov{p}}P,\crG_{\ov{p}}R)\oplus
H_{i-m+1+k} (\crG_{\ov{p}}(\rc L),\crG_{\ov{p}}(\rc L\menos \{\tv\})).\label{equa:directsum}
\end{eqnarray}
In particular, this isomorphism implies that the canonical injection induces a surjective map
$H_{i}(\crG_{\ov{p}}Q,\crG_{\ov{p}}R) \to H_{i}(\crG_{\ov{p}}P,\crG_{\ov{p}}R)$. 
By incorporating this information in the homology long exact sequence of the triple $(\crG_{\ov{p}}P,\crG_{\ov{p}}Q,\crG_{\ov{p}}R)$,
we obtain short exact sequences, for \emph{any} $i$,
\begin{equation}\label{equa:shortandtop}
0
\to
H_{i+1}(\crG_{\ov{p}}P, \crG_{\ov{p}}Q)
\to
H_{i}(\crG_{\ov{p}}Q,\crG_{\ov{p}}R)
\to
H_{i}(\crG_{\ov{p}}P,\crG_{\ov{p}}R)
\to
0.
\end{equation}
Recall $t\leq s$.
As $D\ov{p}$ is a GM-perversity, we have the inequality $D\ov{p}(s+1)
\leq (s-t)+D\ov{p}(t+1)$.
By using it and $i\leq D\ov{p}(s+1)$, we get
$i-m+1+k\leq D\ov{p}(s+1)+t-s+1\leq D\ov{p}(t+1) +1$.
Thus, from \propref{prop:cone}, the isomorphism \eqref{equa:directsum} becomes
$$H_{i}(\crG_{\ov{p}}Q,\crG_{\ov{p}}R)
\cong
 H_{i}(\crG_{\ov{p}}P,\crG_{\ov{p}}R)\quad \text{for} \quad  i\leq D\ov{p}(s+1).$$
From \eqref{equa:shortandtop}, we deduce
\begin{equation}\label{equa:pq}
H_{i}(\crG_{\ov{p}}P,\crG_{\ov{p}}Q)=0\quad \text{for} \quad i\leq D\ov{p}(s+1)+1.
\end{equation}
Let us write $h_{j}$, for $j=1,\,2,\,3$, the homomorphisms induced by the homeomorphism $h$ of the diagram 
\eqref{equa:XandX}, between the homology long exact sequences of
the pairs $(\crG_{\ov{p}}P,\crG_{\ov{p}}Q)$ and $(\crG_{\ov{p}}P',\crG_{\ov{p}}R')$,
$$\xymatrix{
\dots\ar[r]
&
H_{i+1}(\crG_{\ov{p}}P',\crG_{\ov{p}}Q')\ar[r]\ar[d]^-{(h_{3})_{i+1}}
&
H_{i}(\crG_{\ov{p}}Q')\ar[r]\ar[d]^-{(h_{1})_{i}}
&
H_{i}(\crG_{\ov{p}}P')\ar[r]\ar[d]^-{(h_{2})_{i}}
&
H_{i}(\crG_{\ov{p}}P',\crG_{\ov{p}}Q')\ar[d]^-{(h_{3})_{i}}\ar[r]
&\dots \\
\dots\ar[r]
&
H_{i+1}(\crG_{\ov{p}}P,\crG_{\ov{p}}Q)\ar[r]
&
H_{i}(\crG_{\ov{p}}Q)\ar[r]
&
H_{i}(\crG_{\ov{p}}P)\ar[r]
&
H_{i}(\crG_{\ov{p}}P,\crG_{\ov{p}}Q)\ar[r]
& \dots
}$$
By hypothesis, the map $(h_{1})_{i}$ is an isomorphism for any $i$.
For $i\leq D\ov{p}(s+1)+1$, we know that $H_{i}(\crG_{\ov{p}}P',\crG_{\ov{p}}Q')=0$ from \propref{prop:cone} and
$H_{i}(\crG_{\ov{p}}P,\crG_{\ov{p}}Q)=0$ from \eqref{equa:pq}. This implies, with the five lemma, that
$(h_{2})_{i}$ is an isomorphism for any $i\leq D\ov{p}(s+1)$ and a surjection for $i=D\ov{p}(s+1)+1$.
We can apply \lemref{lem:quillenfini} to $\crG_{\ov{p}}P'\to \crG_{\ov{p}}P$ and deduce
$$\pi_{i}(\crG_{\ov{p}}P',v)\cong \pi_{i}(\crG_{\ov{p}}P,h(v))
\; \text{for} \; i\leq D\ov{p}(s+1)
\; \text{and each basepoint}\,v.$$

\smallskip\noindent
$\bullet$ From $D\ov{p}(s+1)+1\geq D\ov{p}(t+1)+1$,
 Propositions~\ref{prop:homotopyxR} and \ref{prop:coneandhomotopy} imply
 $\pi_{i}(\crG_{\ov{p}}P',v)=\pi_{i}(\crG_{\ov{p}}P,v)=0$,
for all $i\geq D\ov{p}(s+1)+1$ and each basepoint $v$.

Combining the two cases, we get the expected result \eqref{equa:conclusioninvariancenew}.
\end{proof}
%
\begin{proof}[Proof of \thmref{thm:homotopyinvariance}]
 All the CS sets that appear in the rest of the proof are supposed to be without exceptional stratum.
From \thmref{thm:thepione}, we know that $\nu$ induces an isomorphism between the fundamental groups,
$\pi_{1}^{\ov{p}}(X,x_{0})\cong \pi_{1}^{\ov{p}}(X^*,x_{0})$ for any regular point $x_{0}\in X$. From \remref{exam:pi0} and the hypothesis, we have
$\pi_{0}^{\ov{p}}(X)=\pi_{0}(X\menos\Sigma_{X})\cong \pi_{0}(X\menos\Sigma_{X^*})\cong \pi_{0}^{\ov{p}}(X^*)$.
Let $U$ be an open subset of $X$, as the regular parts of $U$ and $U^*$ coincide, 
we also get $\pi_{0}^{\ov{p}}(U)\cong \pi_{0}^{\ov{p}}(U^*)$ and
$\pi_{1}^{\ov{p}}(U,x_{0})\cong \pi_{1}^{\ov{p}}(U^*,x_{0})$ for any regular point $x_{0}\in U$.

We  repeat an argument of  King in \cite[Section 3]{MR800845}
(see also \cite[Theorem 5.5.1]{FriedmanBook}) using
an induction on the depth of $X$. The result is obviously true for CS sets of depth 0.
We say that the CS set $X$ has the property $\cW$ if the map
$\pi_{j}^{\ov{p}}\nu\colon\pi_{j}^{\ov{p}}(X,x_{0})\to \pi_{j}^{\ov{p}}(X^*,x_{0})$
is an isomorphism for any $j$ and any regular point $x_{0}$. 
We consider the following properties.
\begin{itemize}
\item[$P(\ell)$:] The CS sets of depth $\leq \ell$ have property $\cW$.
\item[$Q(\ell)$:] The CS sets of the form $M \times \rc W$, with $M$ a trivially filtered manifold 
and $W$ a compact filtered space of  depth $\leq \ell$,  have property $\cW$.
\item[$R(\ell)$:] The CS sets of the form $\R^k \times \rc W$, with $\R^k$  trivially filtered  
and $W$ a compact filtered space of  depth $\leq \ell$,  have property $\cW$.
\end{itemize}
We will show,
$P(\ell) \Rightarrow R(\ell)$,
$R(\ell) \Rightarrow Q(\ell)$
and 
$P(\ell) \wedge Q(\ell) \Rightarrow P(\ell+1)$,
which gives the proof.

\medskip
$R(\ell) \Rightarrow Q(\ell)$. 
We know (\cite[Lemma 2]{MR800845})  that there is a coarsening $Z$ of $\rc W$ such that 
$(M\times \rc W)^*\cong M\times Z$ and $(\R^k\times \rc W)^*\cong \R^k\times Z$.
From \propref{prop:homotopyxR}, 
we deduce that $\crG_{\ov{p}}((\R^k\times \rc W)^*)$ is of the homotopy type of $\crG_{\ov{p}}Z$
and that $\crG_{\ov{p}}(\R^k\times \rc W)$ is of the homotopy type of $\crG_{\ov{p}}(\rc W)$.
The conclusion comes from a new application of \propref{prop:homotopyxR} to
$M\times \rc W$ and $M\times Z$.

\medskip
$P(\ell) \Rightarrow R(\ell)$. This is \propref{prop:invarianceopensubsetnew}.

\medskip
$P(\ell) \wedge Q(\ell) \Rightarrow P(\ell+1)$. 
If $U_{i}$ is an increasing sequence of open subsets of $X$ such that 
$\pi_{*}^{\ov{p}}(U_{i})\cong \pi_{*}^{\ov{p}}(U_{i}^*)$, for any $i$,
then the classic argument of compactness (see \cite[Proposition 15.9]{MR0402714}) gives an isomorphism
$\pi_{*}^{\ov{p}}(\cup_{i}U_{i})\cong \pi_{*}^{\ov{p}}((\cup_{i}U_{i})^*)$.
With the properties on fundamental perverse groups and connected components recalled at the beginning of this proof, 
we can use \lemref{lem:quillen} and are reduced to prove that
$\crG_{\ov{p}}\nu\colon \crG_{\ov{p}}X\to \crG_{\ov{p}}X^*$ induces an isomorphism in
homology for any local coefficient system, $\cE$, on $\crG_{\ov{p}}X^*$.  
As we have established a Mayer-Vietoris exact sequence with coefficients in $\cE$ in \thmref{thm:MV}, the presentation 
using Zorn's Lemma, made by G.Friedman in \cite[Page 257]{FriedmanBook},
 can be reproduced verbatim here.
\end{proof}

\section{Intersection Hurewicz theorem}\label{sec:hur}
Let $(X,\ov{p})$ be a perverse space and $x_{0}\in X$ be a regular point.
Any $\ov{p}$-full simplex of $X$ being a chain of $\ov{p}$-intersection, we have a canonical map,
$\cJ^{\ov{p}}_{i}\colon H_{i}(\crG_{\ov{p}}X;\Z)\to H_{i}^{\ov{p}}(X;\Z)$,
induced by the inclusion of chain complexes, $C_{*}(\crG_{\ov{p}}X;\Z)\to C_{*}^{\ov{p}}(X;\Z)$.
By definition, we call \emph{$\ov{p}$-intersection Hurewicz homomorphism}, the composition
\begin{equation}\label{equa:defhur}
h^{\ov{p}}_{*}\colon \pi^{\ov{p}}_{*}(X,x_{0})\to \tilde{H}_{*}^{\ov{p}}(X;\Z)
\end{equation}
of $\cJ^{\ov{p}}_{*}$ with the Hurewicz homomorphism,  
$h_{*,\crG_{\ov{p}}X}\colon \pi_{*}(\crG_{\ov{p}}X,x_{0})\to \tilde{H}_{*}(\crG_{\ov{p}}X;\Z)$,
of $\crG_{\ov{p}}X$.
This section is devoted to the proof of the following result.

\begin{theoremb}\label{thm:hurintersection}
Let $(X,\ov p)$ be a $\ov{p}$-connected perverse CS set. 
Then the following properties are verified.
\begin{enumerate}[1)]
\item The intersection Hurewicz homomorphism, 
$h_{1}^{\ov{p}}\colon \pi_{1}^{\ov{p}}(X,x_{0})\to \widetilde{H}_{1}^{\ov{p}}(X;\Z)$, 
is isomorphic to the abelianisation $\pi_{1}^{\ov{p}}(X,x_{0})\to \pi_{1}^{\ov{p}}(X,x_{0})^{\ab}$,
 for any regular point $x_{0}$ of $X$.
 \item Let $k\geq 2$. We suppose  
$\pi_{j}^{\ov{p}}(X)=\pi_{j}^{\ov p}(L)=0$ for every link  $L$ of $X$, and each $j\leq k-1$.
Then, the intersection Hurewicz homomorphism
$h_{j}^{\ov{p}}\colon \pi_{j}^{\ov{p}}(X,x_{0})\to \widetilde{H}_{j}^{\ov{p}}(X;\Z)$
 is an isomorphism for $j\leq k$ and a surjection for $j=k+1$.
\end{enumerate}
\end{theoremb}

We already know that the $\ov{p}$-intersection homology of a filtered space, $X$, 
can be different from  the homology of $\crG_{\ov{p}}X$, see \exemref{exam:gajernotintersection}.
Taking in account the (classic) Hurewicz  theorem of $\crG_{\ov{p}}X$, \thmref{thm:hurintersection} above is proven 
if we establish 
a convenient relationship between the homologies  $H_{*}(\crG_{\ov{p}}X;\Z)$ and $H_{*}^{\ov{p}}(X;\Z)$.
As the proof requires different techniques, we split it into two parts. 
Let us begin with the two first homologies.

\begin{theoremb}\label{InterHomolGajer01}
Let $(X,\ov p)$ be a   perverse space.  The  inclusion of  chain complexes induces
isomorphisms in homology,  in degrees~0 and~1, 
\begin{enumerate}[(a)] 
\item $
\cJ_{0}^{\ov{p}} \colon H_0 (\crG_{\ov{p}}X;\Z)\stackrel{\cong}{\longrightarrow} H_0^{\ov p}(X;\Z),
$
\item $
\cJ_{1}^{\ov{p}} \colon H_1 (\crG_{\ov{p}}X;\Z)\stackrel{\cong}{\longrightarrow} H_1^{\ov p}(X;\Z).
$
\end{enumerate}
\end{theoremb}

In general, $\cJ_{2}^{\ov{p}}$ is not an isomorphism, as shows the torus in \exemref{exam:gajernotintersection}.
Let us emphasize that we are dealing with a general perversity $\ov{p}$, without any restriction on it. 
(For instance, the global argument is simpler with the hypothesis $\ov{p}\leq \ov{t}$.) 
We begin with a first technical lemma.

\begin{lemma}\label{*tau}
Let $(X,\ov p)$ be a perverse space.  
Let $\sigma_0, \sigma_1 \colon[0,1] \to X$ be two $\ov p$-allowable simplexes  with $\sigma_0(1) = \sigma_1(0)$
and $\sigma_{0}(0),\,\sigma_{1}(1)\in C_0^{\ov p}(X)$. 
We consider the simplex  $\sigma_0 * \sigma_1 \colon [0,1] \to X$ defined by 
 $$
    \sigma_0 * \sigma_1(t) = \left\{
    \begin{array}{ll}
    \sigma_0(2t) ,& \hbox{if } t\leq 1/2,\\
      \sigma_1(2t-1) ,& \hbox{if } t\geq 1/2.
      \end{array}
      \right.
      $$
Then, 
there exists a $\ov p$-intersection 2-chain  $\tau$ verifying 
\begin{equation}\label{sigma*}
\partial \tau =     \sigma_0 + \sigma_1 - \sigma_0*\sigma_1.
\end{equation}
\end{lemma}

\begin{proof}
This is a well-known argument, the  delicate point is the verification of the $\ov{p}$-allowability of the auxiliary simplexes and of their boundary.
First, the 1-simplex $\ov \sigma_0$ defined by $\ov \sigma_0(t) = \sigma_0(1-t)$ is $\ov p$-allowable since $\sigma_{0}$ is so.
The 1-simplex $\sigma_{0}*\sigma_{1}$ is also $\ov{p}$-allowable since,  for each singular stratum $S$, we have:
      $$
      \dim (\sigma_0*\sigma_1)^{-1}S \leq \max ( \dim \sigma_0^{-1}S, \dim  \sigma_1^{-1}S) \leq 1 - D\ov p(S) -2.
      $$
For any $j\in \N$, we consider $j$-simplexes, constant on $\sigma_{0}(0)$,  $\eps_j \colon \Delta^j \to X$.
They are $\ov p$-allowable if, for each singular stratum $S$ with $\sigma_0(0) \in S$,  we have
$ \dim \eps_j  ^{-1}S \leq j - D\ov p(S) -2$, which is a consequence of
 $\sigma_0(0) \in C_0^{\ov p}(X)$. 
 Since $\partial \eps_j = 0$ or  $\eps_{j-1}$, the simplexes $\eps_j$ are of $\ov p$-intersection.

 Let $a_{0},a_{1},a_{2}$ be the vertices of $\Delta^2$. The 2-simplex, $\beta\colon \Delta^2\to X$, defined by
 $\beta(t_{0}a_{0}+t_{1}a_{1}+t_{2}a_{2})=\sigma_{0}(t_{1})$ is $\ov{p}$-allowable since,
for each singular stratum $S$, we have
$$
\dim \beta^{-1}S \leq 1+\dim \sigma_0^{-1} S
\leq
2 - D\ov p(S) -2 .
$$
A direct computation gives
$\partial \beta =  \sigma_0  - \eps_1 + \ov \sigma_0$  which is a $\ov p$-allowable chain. 
Let $f\colon \Delta^2\to [0,1]$ be the linear map such that $f(a_{0})=1$, $f(a_{1})=0$ and $f(a_{2})=1/2$.
The 2-simplex, $\tau\colon \Delta^2\to X$, defined by $\tau=(\sigma_{0}\ast\sigma_{1})\circ f$, is also $\ov{p}$-allowable since,
for each singular stratum $S$, we have
$$
\dim \tau^{-1}S =\dim f^{-1}\left( (\sigma_0*\sigma_1)^{-1} S\right) \leq 1+ \dim (\sigma_0*\sigma_1)^{-1} S \leq 2 - D\ov p(S) -2.
$$
A computation gives $\partial \tau = \ov \sigma_0  -  \sigma_1 +  \sigma_0 * \sigma_1$
which is a $\ov{p}$-allowable chain.
In summary, the  2-chain $\gamma =  \eps_2 + \beta - \tau $ is of $\ov p$-intersection and verifies
$$\partial \gamma =\partial \eps_2 + \partial \beta -  \partial \tau = \eps_1 +(   \sigma_0  - \eps_1 + \ov \sigma_0) 
-  \ov \sigma_0  +  \sigma_1 - \sigma_0 * \sigma_1 =  
  \sigma_0
  + \sigma_1 -  \sigma_0 * \sigma_1.$$
\end{proof}

\begin{proof}[Proof of \thmref{InterHomolGajer01}]
The equality $C_0^{\ov p}(X) = C_0(\crG_{\ov p} X)$ gives  the surjectivity of $\cJ_{0}^{\ov{p}}$.
 \emph{Let us study the injectivity of $\cJ_{0}^{\ov{p}}$ and the surjectivity of $\cJ_{1}^{\ov{p}}$.}
 We get the claim if for any chain $ \eta \in C_1^{\ov p}(X)$   of this type,
with $ \partial \eta \in C_0(\crG_{\ov p}X) = C_0^{\ov p}(X)$,  we find 
$\alpha \in  C_1(\crG_{\ov p}X)$ with 
$\eta  - \alpha  \in \partial C_2^{\ov p}(X)$.
 This chain can be uniquely written 
 $\eta= \sum_{i\in I} n_i \sigma_i$ with $n_i = \pm 1$.

 As $ \eta \in C_1^{\ov p}(X)$, the chain $\eta$ and its boundary $\partial \eta$ are $\ov{p}$-allowable.
There can exist simplexes $\sigma_{i}$ with a non-$\ov{p}$-allowable face but this face must then be eliminated in the 
boundary $\partial \eta$. 
Unfortunately such a simplex is not accepted in a Gajer space. 
We must therefore substitute them by simplexes with $\ov{p}$-allowable faces.
 Using a subdivision, 
 we can suppose that each $\sigma_{i}$, $i\in I$, has at most one  vertex 
 which is not $\ov{p}$-allowable. 
Thus, we can write 
 $$
 \eta= \sum_{i\in I_0} n_i \sigma_i + 
  \sum_{i\in I_1} n_i \sigma_i=\eta_{0}+\eta_{1},
$$
 where $\sigma_{i}(0),\, \sigma_{i}(1)\in C_{0}^{\ov{p}}(X)$  if $i\in I_{0}$ and exactly one of these two
 points is $\ov{p}$-allowable  if $i\in I_{1}$.
By definition, we have $ \eta_0\in C_1(\crG_{\ov p}X)$.
 Let  $k\in I_1$ and suppose for simplicity that  $\sigma_k(1) \not \in C_0^{\ov p}(X)$.
 Since  $\partial \eta_{1} \in C_0^{\ov p}(X)$,  there exists $j\in I_1$ with
 $\sigma_k(1)= \sigma_j(0)$ and $n_k =  n_j$.
 The chain $\eta_{1}$ is a sum of chains $n_k \sigma_k +n_j\sigma_j $ of this type. 
 Following \lemref{*tau},
 we have $ \sigma_k + \sigma_j  -  \sigma_k*\sigma_j \in  \partial  C_2^{\ov p}(X)$ and 
 $\sigma_k * \sigma_j \in C_1(\crG_{\ov p} X)$.
 
\medskip
Let's get to the  difficult point: \emph{the map $\cJ_{1}^{\ov{p}}$ is a monomorphism.}
Given $ \eta \in C_2^{\ov p}(X)$  with $ \partial \eta \in C_1(\crG_{\ov p}X)$  we have to find $\alpha  \in C_2(\crG_{\ov p}X)$   with 
$\partial \eta  =  \partial \alpha$.
Let's start from a $\ov{p}$-allowable 2-simplex $\Delta^2\to X$, and apply to it the pseudobarycentric subdivision 
recalled in \remref{rem:pseusobarycenter}. 
By construction, all the simplexes 
containing the pseudobarycentre of $\Delta^2$ are $\ov{p}$-allowable. 
This means that  all the 1-faces of the new ``small'' simplexes are $\ov{p}$-allowable,   except at most one.
Also, all the vertices are $\ov{p}$-allowable except at most two.
Thus, any chain $ \eta \in C_2^{\ov p}(X)$ can  be uniquely written 
$$
 \eta= \sum_{i\in J_{1}} n_i \sigma_i +  \sum_{i\in J_{2}} n_i \sigma_i,
$$
 with $n_i = \pm 1$, all the 1-faces of the $\sigma_{i}$ with $i\in J_{1}$ are $\ov{p}$-allowable and the $\sigma_{i}$ with $i\in J_{2}$ have exactly two $\ov{p}$-allowable 1-faces.
 We proceed in two steps, beginning with the 1-faces.

 \smallskip
 $\bullet$ \emph{First step: Cancelation of the bad 1-faces of $J_{2}$}.  
 Let $\sigma_{k}\colon \langle a_{0},a_{1},a_{2}\rangle \to X$ with $k\in J_{2}$.
 Suppose that the restriction $\tau$ of $\sigma_{k}$ to $\langle a_{0},a_{1}\rangle$ is the bad face.
 So, there exists $\sigma_{j}\colon \langle a_{0},a_{1},a_{3}\rangle \to X$, $j\in J_{2}$, in such a way that
$\tau$ does not appear in the boundary of $n_k \sigma_k +n_j\sigma_j$.
We proceed to a pseudobarycentric subdivision of $\sigma_{k}$, of pseudobarycentre $b$,
see \remref{rem:pseusobarycenter}. 
By construction, the restrictions
to $\langle a_{0},b\rangle$ and $\langle a_{1},b\rangle$ are $\ov{p}$-allowable.
Let us define two new simplexes $\sigma'_k \colon \Delta^2 \to X$ and $\sigma'_j \colon \Delta^2 \to X$ as indicated in the figure below

\hskip 2.6cm
\scalebox{0.4}
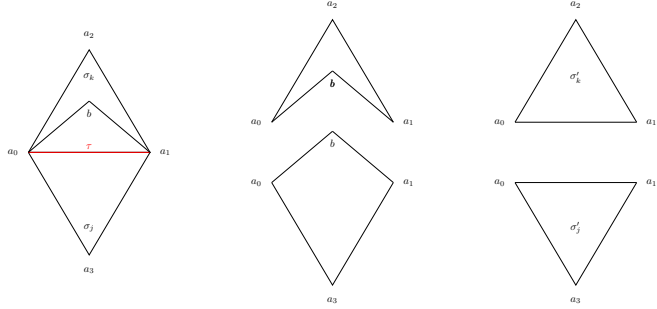
{
\begin{tikzpicture}

\draw  (-2,0)-- (2,0) -- (0,3.4)-- (-2,0);
\draw  (-2,0)-- (2,0) -- (0, -3.4)-- (-2,0);
\draw [red,thick] (-2,0)-- (2,0) ;

\draw (-2.5,0) node {$a_0$};
\draw (2.5,0) node {$a_1$};
\draw (0,3.9) node {$a_2$};
\draw (0,-3.9) node {$a_3$};

\draw (0,2.5) node {$\sigma_k$};
\draw[red] (0,.2) node {$\tau$};
\draw (0,-2.5) node {$\sigma_j$};

\draw [black,thick] (0,1.7)-- (-2,0);
\draw [black,thick] (0,1.7)-- (2,0);
\draw (0,1.3) node {$b$};

\draw  (10,1) -- (8,4.4)-- (6,1);

\draw [black,thick] (8,2.7)-- (6,1);
\draw [black,thick] (8,2.7)-- (10,1);

\draw (5.5,1) node {$a_0$};
\draw (10.5,1) node {$a_1$};
\draw (8,4.9) node {$a_2$};
\draw (8,2.3) node {$b$};

\draw [black,thick] (8,.7)-- (6,-1);
\draw [black,thick] (8,.7)-- (10,-1);
\draw  (10,-1)-- (8, -4.4)--(6,-1);

\draw (8,.3) node {$b$};
\draw (5.5,-1) node {$a_0$};
\draw (10.5,-1) node {$a_1$};
\draw (8,-4.9) node {$a_3$};

\draw  (18,1) -- (16,4.4)-- (14,1) ;
\draw[black,thick]  (18,1) -- (14,1) ;

\draw (13.5,1) node {$a_0$};
\draw (18.5,1) node {$a_1$};
\draw (16,4.9) node {$a_2$};
\draw (8,2.3) node {$b$};

\draw (16,2.5) node {$\sigma'_k$};
\draw[black,thick]  (18,-1) -- (14,-1) ;
\draw  (18,-1)-- (16, -4.4)--(14,-1);

\draw (13.5,-1) node {$a_0$};
\draw (18.5,-1) node {$a_1$};
\draw (16,-4.9) node {$a_3$};

\draw (16,-2.5) node {$\sigma'_j$};

\end{tikzpicture}
}

\noindent So, by construction, the 2-simplexes $\sigma'_k,\sigma'_j$ and all their 1-faces are $\ov p$-allowable
and we also have 
$\partial (n_k \sigma'_k +n_j\sigma'_j )= \partial (n_k \sigma_k +n_j\sigma_j)$.
The 2-chain 
$$
\eta' = \eta - (n_k \sigma_k +n_j\sigma_j ) + (n_k \sigma'_k +n_j\sigma'_j) \in C_2^{\ov p}(X)
$$
 verifies $\partial \eta'  = \partial \eta$ and the cardinal of the corresponding subset $J'_{2}$ is strictly smaller than that  of $J_{2}$.
 With iterations, we get a 2-chain $\alpha=\sum_{i\in K}n_{i}\sigma_{i}$, with $n_{i}=\pm 1$ and all the faces of the $\sigma_{i}$ $\ov{p}$-allowable except possibly two vertices.

  \medskip
$\bullet$ \emph{Second step: the vertices}. 
We start with the previous $\alpha$. 
Proceeding as before  to a pseudobarycentric subdivision, we can suppose that all the faces are $\ov{p}$-allowable, except possibly one vertex.
Therefore, we decompose
$$
 \alpha= \sum_{i\in K_{1}} n_i \sigma_i +  
\sum_{i\in K_{2}} n_i \sigma_i=\alpha_{1}+\alpha_{2}
$$
 where 
 all the faces of the $\sigma_{i}$ with $i\in K_{1}$ are $\ov{p}$-allowable and all the faces of the 
 $\sigma_{i}$ with $i\in K_{2}$ are $\ov{p}$-allowable except one vertex.
 Let  $i\in K_{2}$ and  suppose that $\sigma_{i}(a_{0})$ is the only not $\ov p$-allowable vertex.
Recall $\partial\alpha=\partial\eta\in    C_1(\crG_{\ov p} X )$. Thus $\sigma_{i}(a_{0})$ does not belong to $\partial \alpha$ and
there exists  $k \in K_{2}$ such that the restrictions of $\sigma_{i}$ and $\sigma_{k}$ to the face $\langle a_{0},a_{1}\rangle$ verifies
 $$n_{i} {\sigma_{i}}_{|<a_0,a_1>} + n_{k} {\sigma_{k}}_{|<a_0,a_1>}=0.$$
(We can have a cancellation with other faces of $\sigma_{k}$ but it is the same pattern with a sign adjustment.)

\begin{multicols}{2}
By repeating this  process, we construct a chain 
\begin{equation}\label{beta}
\beta = \sum_{k=0}^m n_{i_k} \sigma_{i_k} ,
\end{equation}
with $\{i_0, \ldots,i_m\} \subset K_{2}$ and $\partial \beta \in C_1(\crG_{\ov p}X)$.
This chain appears as a map still denoted by $\beta\colon K_{\beta}\to X$ from a polyhedron $K_{\beta}$ as in the figure beside. 
The letters $b_{i}$ are a notation for the vertices of $\Delta^2=\langle a_{0},a_{1},a_{2}\rangle$.

The illustrated situation corresponds to an annulation of the  vertex $ \sigma_{i_0}(a_0)$ after an iteration  of 4  steps. The case of 0 iteration corresponds to a disk.
In general,  we get a similar figure after  a finite number $m$ of steps.

\hskip -5cm
\scalebox{0.8}{
\definecolor{zzttqq}{rgb}{0.6,0.2,0}
\definecolor{ududff}{rgb}{0.30196078431372547,0.30196078431372547,1}
\definecolor{xdxdff}{rgb}{0.49019607843137253,0.49019607843137253,1}
\definecolor{qqqqff}{rgb}{0,0,1}
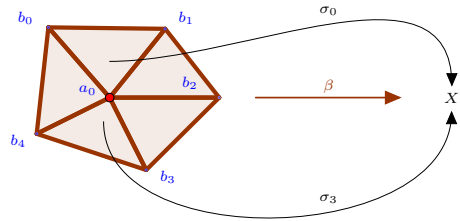
\begin{tikzpicture}[line cap=round,line join=round,>=triangle 45,x=.6cm,y=.6cm]
\clip(-12.346215761837963,-4) rectangle (17.97224084001275,6.0098305438753075);
\fill[line width=2pt,color=zzttqq,fill=zzttqq,fill opacity=0.10000000149011612] (0,0) -- (3,0) -- (1.52,1.9) -- cycle;
\fill[line width=2pt,color=zzttqq,fill=zzttqq,fill opacity=0.10000000149011612] (-1.68,1.94) -- (0,0) -- (-2,-1) -- cycle;
\fill[line width=2pt,color=zzttqq,fill=zzttqq,fill opacity=0.10000000149011612] (-2,-1) -- (1,-2) -- (0,0) -- cycle;
\fill[line width=2pt,color=zzttqq,fill=zzttqq,fill opacity=0.10000000149011612] (-1.68,1.94) -- (1.52,1.9) -- (0,0) -- cycle;
\fill[line width=2pt,color=zzttqq,fill=zzttqq,fill opacity=0.10000000149011612] (3,0) -- (0,0) -- (1,-2) -- cycle;
\draw [line width=2pt,color=zzttqq] (0,0)-- (3,0);
\draw [line width=2pt,color=zzttqq] (3,0)-- (1.52,1.9);
\draw [line width=2pt,color=zzttqq] (1.52,1.9)-- (0,0);
\draw [line width=2pt,color=zzttqq] (-1.68,1.94)-- (0,0);
\draw [line width=2pt,color=zzttqq] (0,0)-- (-2,-1);
\draw [line width=2pt,color=zzttqq] (-2,-1)-- (-1.68,1.94);
\draw [line width=2pt,color=zzttqq] (-2,-1)-- (1,-2);
\draw [line width=2pt,color=zzttqq] (1,-2)-- (0,0);
\draw [line width=2pt,color=zzttqq] (0,0)-- (-2,-1);
\draw [line width=2pt,color=zzttqq] (-1.68,1.94)-- (1.52,1.9);
\draw [line width=2pt,color=zzttqq] (1.52,1.9)-- (0,0);
\draw [line width=2pt,color=zzttqq] (0,0)-- (-1.68,1.94);
\draw [line width=2pt,color=zzttqq] (3,0)-- (0,0);
\draw [line width=2pt,color=zzttqq] (0,0)-- (1,-2);
\draw [line width=2pt,color=zzttqq] (1,-2)-- (3,0);
\begin{scriptsize}
\draw [fill=red] (0,0) circle (2pt);
\draw[color=qqqqff] (-2.3,2.2) node {$b_0$};
\draw[color=qqqqff] (2.1,2.2) node {$b_1$};
\draw[color=qqqqff] (2.2,0.4) node {$b_2$};
\draw[color=qqqqff] (1.6,-2.2) node {$b_3$};
\draw[color=qqqqff] (-2.5,-1.2) node {$b_4$};
\draw[color=zzttqq] (6,.3) node {$\beta$};
\draw[->, color=zzttqq, thick]  (4,0)-- (8,0);

 \node[anchor=east] at (0,1) (principio) {};
  \node[anchor=west] at (9,0) (fin) {$X$};
  \draw (principio) edge[out=0,in=90,->] (fin);
  \draw[color=qqqqff] (-.6,.2) node {$a_0$};
    \draw[] (6,2.4) node {$\sigma_0$};
    
     \node[anchor=east] at (0,-.5) (principiobis) {};
  \node[anchor=west] at (9.2,-.2) (finbis) {}{};
  \draw (principiobis) edge[out=-90,in=-90,->] (finbis);
    \draw[] (6,-2.8) node {$\sigma_3$};
    
        \draw[color=zzttqq] (-4,0.5) node {$K_\beta$};


\draw [color=xdxdff] (3,0) circle (0.5pt);
\draw [color=ududff] (1.52,1.9) circle (0.5pt);
\draw [color=ududff] (-1.68,1.94) circle (0.5pt);
\draw [color=ududff] (-2,-1) circle (0.5pt);
\draw [color=ududff] (1,-2) circle (0.5pt);
\end{scriptsize}
\end{tikzpicture}
}
\end{multicols}

The chain $\alpha$ is the sum of $\alpha_1 \in C_2(\crG_{\ov p}X)$ and some chains of type
\eqref{beta}. So, we get the claim if we  find $\beta' \in C_2(\crG_{\ov p}X)$ with $\partial \beta = \partial \beta'$.
This new chain $\beta'$ is obtained from $\beta$ slightly moving the vertex $a_0$ in $a'_{0}$, keeping the $b_{i}$ unchanged, and corresponding to a new triangulation of 
the polyhedron $K_{\beta}$.
We denote by $\sigma'_{i}$ the simplex corresponding to $\sigma_{i}$ in this new triangulation. 
%
We need to prove that each $\sigma'_{i_k}$ is a $\ov p$-full simplex for a convenient choice of $a'_0$. 
Since $\partial \eta \in C_1(\crG_{\ov p} X)$ then  it suffices to prove that all the faces of the simplexes $\sigma'_{i_k}$ meeting $\beta (a'_0)$ are $\ov p$-allowable.
For that, we adapt to the case of polyhedra the proof of \cite[Proposition 6]{CST8} made for simplexes.
We distinguish the three possible dimensions.

 \smallskip
 \emph{1) The 2-dimensional faces}. The simplexes $\sigma'_{i_k}$  are $\ov p$-allowable  for any choice of $a'_{0}$, since we have, for any singular stratum,
\begin{eqnarray*}
\max \{ \dim \sigma'^{-1}_{i_k}S \mid 0\leq k\leq  m \} &=&
\dim \beta^{-1}S \\
&=&  \max \{ \dim \sigma^{-1}_{i_k}S \mid 0\leq k\leq  m \} \}\\
&  \leq & 2 -D\ov p(S) -2,
\end{eqnarray*}
Notice that $\beta^{-1}S = \emptyset$ for each singular stratum $S \in \cS_X$ with $D\ov p(S)>0$.

\smallskip
Before studying the 0 and 1-dimensional faces, we establish some notation. Since the simplexes of $\beta$,
as well as their faces, are $\ov{p}$-allowable, we get the following properties (cf. \defref{def:homotopygeom}).

$\bullet$ The subset $D=\cup\{\beta^{-1}S\mid D\ov{p}(S)=0\}$
is a finite subset included in $\inte K_{\beta}=K_{\beta}\menos \partial K_{\beta}$.

$\bullet$ The subset $E=\cup\{\beta^{-1}S\mid D\ov{p}(S)=-1\}$
is included in $K_{\beta}$ and its dimension is smaller than~1. We choose a polyhedron $P$
such that $E\subset P$ and $\dim P\leq 1$.

\smallskip 
\emph{2) The 0-dimensional faces}.  
All vertices, except $a'_{0}$, belong to the boundary and are therefore $\ov{p}$-allowable.
The restriction of $\beta$ to $\langle a'_{0}\rangle$ are 
 $\ov p$-allowable if $a'_{0}\notin D\cup P$. By dimensional reasons, the open set
 $\cO
=\inte K_{\beta}\menos (D\cup P)$ is dense in $\inte K_{\beta}$. If we choose $a'_{0}$
in this subset (which is not empty!) we get the $\ov{p}$-allowability of all the 0-dimensional faces 
of the simplexes $\sigma'_{i_{k}}$.

\smallskip
 \emph{3) The 1-dimensional faces}.  
 The one-dimensional faces   appearing in $\partial \beta$ are $\ov p$-allowable. Thus, we are reduced to the 1-simplexes obtained by the restriction of $\beta$ to
 $\Delta^1= \langle a'_0,b_i\rangle$, with $0\leq i\leq m$.
 They are  $\ov p$-allowable if the following conditions hold, for any $i$,
 \begin{equation}\label{S=0-1}
 \left\{
 \begin{array}{l}
 \text{(i)}\quad D\cap \langle a'_{0},b_{i}\rangle=\emptyset, \quad\text{and}\\
\text{(ii)}\quad P\cap \langle a'_{0},b_{i}\rangle \quad\text{is a finite set.}
 \end{array}
 \right.
 \end{equation}
 Since $D$ is a finite set, the family of points $a'_{0}\in \inte K_{\beta}$ verifying (i) are a dense open subset
 $\cO'$ of $\inte K_{\beta}$. Notice that $\cO\cap \cO'$ is also a dense open subset of $\inte K_{\beta}$.
 Applying the general position principle (\cite[Section 5.34]{MR0350744}), 
 we find $a'_{0}\in \cO\cap \cO'$ verifying (ii) for any $i$.
 This gives \eqref{S=0-1}.

\smallskip
 We have proven that all  faces of the simplexes $\sigma'_{i_k}$ are $\ov p$-allowable. We obtain
$\beta' \in C_2(\crG_{\ov p}X)$ and $\partial \beta = \partial \beta'$ which ends the proof.
\end{proof}

We continue our study of the relationship between the homologies $H_*(\crG_{\ov p} X)$ and $H_*^{\ov p}(X)$
by looking to higher degrees. 
As the links are supposed to be $\ov{p}$-connected, 
we can suppress the mention of the basepoint in the rest of this section.

\begin{theoremb}\label{InterHomolGajer}
Let $(X,\ov p)$ be a  perverse CS  set and $k\geq 2$ be an integer. 
We suppose
$\pi_{j}^{\ov p}(L)=0$ for each link  $L$ of $X$  and each $j\leq k-1$. 
Then, the  inclusion of chain complexes, $C_*(\crG_{\ov p}X;\Z) \hookrightarrow C_*^{\ov p}(X;\Z)$, induces
\begin{itemize}
\item[(i)]  an isomorphism:
$
\cJ_{j}^{\ov{p}}  \colon H_j (\crG_{\ov{p}}X;\Z)\stackrel{\cong}{\longrightarrow} H_j^{\ov p}(X;\Z),
$
for $j \leq k$, 
and
\item[(ii)] an epimorphism
$
\cJ_{k+1}^{\ov{p}}  \colon H_{k+1} (\crG_{\ov{p}}X;\Z)\twoheadrightarrow H_{k+1}^{\ov p}(X;\Z).
$
\end{itemize}
\end{theoremb}

Before  the proof, we introduce  a property  adapted to conification and Mayer-Vietoris sequences.

\begin{definition}\label{defBPk}
Let $(X,\ov p)$ be a perverse space and  $k \in \Z$.
We say that  $X$ \emph{verifies the property}  $P_k$, written $P_k(X)$, if the  inclusion,
$C_*(\crG_{\ov p}X;\Z) \hookrightarrow C_*^{\ov p}(X;\Z)$, of chain complexes induces
\begin{itemize}
\item[(i)]  the isomorphism
$
\cJ_{j}^{\ov{p}} \colon H_j (\crG_{\ov{p}}X;\Z)\stackrel{\cong}{\longrightarrow} H_j^{\ov p}(X;\Z),
$
for $j \leq k$, 
and
\item[(ii)]the epimorphism
$
\cJ_{k+1}^{\ov{p}} \colon H_{k+1} (\crG_{\ov{p}}X;\Z)\twoheadrightarrow H_{k+1}^{\ov p}(X;\Z).
$
\end{itemize}
\end{definition}

Notice that \thmref{InterHomolGajer01} gives $P_{0}(X)$.

\begin{proposition}\label{prop:LtocL}
 Let $L$ be a  compact  filtered space 
 and $\rc L$ be the open cone, with the conic filtration.
Let $\ov{p}$ be a   perversity on $\rc L$,  
we also denote  by $\ov{p}$ the perversity induced on $L$.
Let $k\geq  2$ be an integer. 
We suppose $\pi_{j}^{\ov p}(L)=0$ for each $j\leq k-1$. 
Then, the property $P_k(L)$ implies $P_k(\rc L)$.
\end{proposition}

In other words, we  prove that $P_k(\rc L \menos \{ \tv \})$ implies $P_k(\rc L)$, 
cf. \propref{prop:homotopyxR} and \cite[Corollary 1]{CST8}.

\begin{proof} 
Let $\tv$ be the apex of the cone $\rc L$, $\ov p(\tv)=p$ and $D\ov p(\tv) =q$.

(i) With \thmref{InterHomolGajer01},  we can take $2\leq j\leq k$. We have $C_{\leq q+1} (\crG_{\ov p} (\rc L)) = 
C_{\leq q+1} (\crG_{\ov p}(\rc L\menos \{ \tv\})) $ and $C_{\leq q+1}^{\ov p} (\rc L)= C_{\leq q+1}^{\ov p} (\rc L\menos \{ \tv\}) $.
If $j\leq q$, then the map   $\cJ_{j,\rc L}^{\ov{p}}\colon H_{j}(\crG_{\ov{p}}\rc L)\to H_{j}^{\ov{p}}(\rc L)$ 
is isomorphic to the map,
$\cJ_{j,\rc L\menos\{\tv\}}^{\ov{p}}H_j (\crG_{\ov{p}}( \rc L \menos \{ \tv \}))\to H_j^{\ov p}( \rc L \menos \{ \tv \})$
 induced by the inclusion. This is an isomorphism with the hypothesis $P_k( \rc L \menos \{ \tv \})$.

Suppose $ q+1 \leq j $. From $q\leq k-1$ and \propref{prop:coneandhomotopy}, we deduce $\pi^{\ov p}_{j} (\rc L)=0$.
The (classic) Hurewicz theorem implies $\widetilde H_*(\crG_{\ov p}\rc L)= 0$ and the map  $\cJ_{j,\rc L}^{\ov{p}}$ is a monomorphism.
From \cite[Proposition 3]{CST8}, we have $H_{j}^{\ov p} (\rc L)=0$ and the map  
$\cJ_{j,\rc L}^{\ov{p}}$ is an epimorphism.

(ii) We have to prove that $\cJ_{k+1,\rc L}^{\ov{p}}\colon H_{k+1}(\crG_{\ov{p}}(\rc L))\to H_{k+1}^{\ov{p}}(\rc L)$ is an epimorphism.

If $k+1\geq q+1$ then $ H_{k+1}^{\ov p} (\rc L) = 0$  (cf.  \cite[Proposition 3]{CST8} and $k+1\ne 0$) 
and therefore $\cJ_{k+1,\rc L}^{\ov{p}}$ is an epimorphism.
Let us suppose $k+1\leq q$. 
Consider the commutative diagram,
$$
\xymatrix{
\pi_{k+1}^{\ov p} ( L)\ar[r]^{J_2}   \ar[d]_-{h_{k+1, \crG_{\ov{p}}L}}
&
\pi_{k+1}^{\ov p} (\rc L)\ar[rr]^-{h_{k+1,\crG_{\ov{p}}\rc L}} 
&&
 H_{k+1}(\crG_{\ov p}\rc L) \ar[d]^-{\cJ_{k+1,\rc L}^{\ov{p}}}
  \\
    H_{k+1}(\crG_{\ov p}  L) \ar[r]^-{\cJ_{k+1,L}^{\ov{p}}}
    &
     H_{k+1}^{\ov p}( L)  \ar[rr]^-{J_1}
     &&  
   H_{k+1}^{\ov p}(\rc L),\\
}
$$
where the maps $J_{*}$ are induced by the inclusion $L\to \rc L$ and $h_{*,*}$ are the (classic) Hurewicz homomorphisms.
We know that $J_{1}$ is an isomorphism (\cite[Proposition 3]{CST8} and $k+1\leq q$) and that
$\cJ_{k+1,L}^{\ov{p}}$ is an epimorphism grants to $P_{k}(L)$.
From $\pi_{\ell}(\crG_{\ov{p}}L)=0$ for $\ell\leq k-1$ and the classic Hurewicz theorem ($k\geq 2)$ 
we obtain the surjectivity of
$h_{k+1,\crG_{\ov{p}}L}$.
We deduce the surjectivity of $\cJ_{k+1,\rc L}^{\ov{p}}$.
\end{proof}

\begin{proposition}\label{MV}
Let  $(X,\ov p)$ be a perverse space and $k\in \N$. 
For any open covering,  $\{U,V\}$, of $X$, we have
$$
P_k(U), P_k(V), P_k(U\cap V) \Longrightarrow P_k(X).
$$
\end{proposition}
\begin{proof}
It suffices to apply the Five Lemma to  the morphism $\cJ^{\ov{p}}_{*}$,
$$
\scalebox{0.7}{
\xymatrix{
H_j(\crG_{\ov p}(U\cap V) ) \ar[r]  \ar[d]
& H_j(\crG_{\ov p}(U ) )  \oplus H_j(\crG_{\ov p}( V) )  \ar[r]  \ar[d]
& 
H_j(\crG_{\ov p}(X ) ) \ar[r]  \ar[d]
&
H_{j-1}(\crG_{\ov p}(U\cap V) ) \ar[r] \ar[d]
& H_{j-1}(\crG_{\ov p}(U ) )  \oplus H_{j-1}(\crG_{\ov p}( V) ) \ar[d]
\\
H_j^{\ov p}(U\cap V) ) \ar[r] & H_j^{\ov p}(U )   \oplus H_j^{\ov p}( V)   \ar[r] & 
H_j^{\ov p}(X) \ar[r]  
&
H_{j-1}^{\ov p}(U\cap V)  \ar[r]& H_{j-1}(\crG_{\ov p}(U ) )  \oplus H_{j-1}^{\ov p} ( V), 
}
}
$$
between the Mayer-Vietoris  exact sequences (cf. \thmref{thm:MV} and 
\cite[Theorem 1]{CST8}). 
\end{proof}

\begin{proof}[Proof of \thmref{InterHomolGajer}]
We proceed by induction on the depth of the CS sets. If the depth is null, then $\cJ_{j}^{\ov{p}}$ is an isomorphism for each $j\in \N$
 since $H_*(\crG_{\ov p} X) = H_*(X) = H_*^{\ov p} (X)$.
  Let us study the inductive phase. We proceed in two steps.

\emph{First Step: the open subsets of the conical charts.}  
Let $\varphi \colon U \times \rc L \to V$ be a conical chart. 
We define a \emph{cube}  as a product 
$]a_{1},b_{1}[ \times \cdots \times 
]a_{m},b_{m}[ \subset U$, with $a_{\bullet}$, $b_{\bullet}$ $\in\Q$,
 and denote by $\cC$ the family of cubes.
The {\em truncated cone} $\rc_{t}L$ is  the quotient $\rc_{t}L = L\times [0,t[ / L \times \{ 0\}$. 
The following family,
$$
\mathcal{U} = 
\left\{ \varphi(C\times c_{t}L )\mid C \in \cC,
0<t<1, t\in\Q  \right\} \cup 
\left\{ \varphi(C \times L \times ]a,b[ ) \mid
C \in \cC, 
0\leq a < b \leq 1,\,a, b\in\Q
 \right\},
$$
is a countable open basis of $V$ closed by finite intersections. 
By induction hypothesis we have $P_k(L)$. %
From it and \propref{prop:LtocL}, we deduce $P_{k}(\rc_{t}L)$ and thus
$P_{k}(\varphi(C\times \rc_{t}L))$ since $C$ is homeomorphic to $\R^n$ and $\varphi$ is a stratified homeomorphism.
Similarly, from $P_{k}(L)$, we deduce $P_{k}(\varphi(C\times L\times ]a,b[))$.
In summary, we have proven $P_{k}(W)$ for any $W\in \cU$.

Let $\cU_{1}$ be the family of finite unions of elements of $\cU$. We consider $U=\cup_{i=1}^pU_{i}\in\cU_{1}$. 
By induction on $p$, from  $(U_1 \cup \cdots \cup U_{p-1})  \cap U_{p} = (U_1 \cap U_{p+1}  ) \cup \cdots \cup (U_{p-1} \cap U_{p} )
$, we deduce $P_{k}((U_1 \cup \cdots \cup U_{p-1})  \cap U_{p})$.
\propref{MV} implies  $P_k(U)$ for each $U \in \cU_1$.
Notice that $\cU_{1}$ is a countable open basis of $V$ closed by finite intersections.

Let $\cU_{2}$ be the family of numerable unions of elements of $\cU$. We consider $U=\cup_{i\in\N}U_{i}\in \cU_{2}$.
By setting $V_{i}=V_{i-1}\cup U_{i}$ we get $U=\cup_{i\in \N}V_{i}$ where $(V_{i})_{i}$ is an increasing sequence of open subsets in $\cU_{1}$.
A classic argument for homology theories with compact supports gives
$$
H_*(\crG_{\ov p} U) = \varinjlim_iH_*(\crG_{\ov p} V_i)
\ \ \ 
\hbox{ and }
\ \ \ 
H_*^{\ov p}(U)= \varinjlim_i H_*^{\ov p}(V_i).
$$
From $P_{k}(V_{i})$ for all $i\in \N$, we deduce $P_{k}(U)$ for each $U\in\cU_{2}$.
We get the claim since $\cU_{2}$ is  the topology of $V$.

\emph{Second Step:  we prove $P_k(X)$.}
We consider the family $\cV$ of open subsets $V$ of $X$ with $P_{k}(V)$. 
We order $\cV$ by inclusion. 
If $(V_{i})_{i\in I}$ is an increasing family of elements of $\cV$, an argument as above implies 
 $P_{k}(\cup_{i\in I}V_{i})$.
 Therefore, by Zorn's lemma, the family $\cV$ has a maximal element $W$.
We are reduce to show  $W = X$. Let us suppose  the existence of  $x\in X\menos W$. 
Let $\varphi \colon U \times \rc L \to V$ be a conical chart of $x$. 
The first step gives $P_k(V)$ and $P_k(W \cap V)$. We also  have $P_k(W)$ by choice of $W$. 
From \propref{MV}, we deduce $P_k(W\cup V)$.
The maximality of $W$ implies $W=W\cup V$ and $x \in W$. From this contradiction, we deduce $W=X$.
\end{proof}

\begin{proof}[Proof of \thmref{thm:hurintersection}]
The two properties are consequences of the Hurewicz theorem for $\crG_{\ov{p}}X$ and results on the map
$H_j (\crG_{\ov{p}}X;\Z){\longrightarrow} H_j^{\ov p}(X;\Z)$. 
Property 1) is a consequence of \thmref{InterHomolGajer01} and
Property 2) is covered by \thmref{InterHomolGajer}.
\end{proof}

\begin{example}\label{exem:pashurewicz}
We construct a CS set $X$  with $\pi_{j}^{\ov{p}}(X)=0$, for $j\leq 2$, (i.e., $k=3$ in \thmref{InterHomolGajer})
and $\cJ^{\ov{p}}_{4}$ not surjective. Thus the hypotheses on the links cannot be removed in
\thmref{thm:hurintersection}.2). 

We start with the Hopf principal bundle $S^3 \to S^7  \xrightarrow{\tau}  S^4$ and  quotient it
by the $\Z_2$-action to obtain the bundle $\RP^3 \to \RP^7  \xrightarrow{\tau'} S^4$.
Let $\tc \RP^3=\R P^3\times [0,1]/\RP^3\times\{0\}$ be the closed cone.
The total space of the associated cone bundle of $\tau$,  $\tc \RP^3 \to T \xrightarrow{\tau''}  S^4$,
 is  a CS set with $S^4$ as  only singular stratum. 
Finally, we consider the double mapping cylinder of 
$\tau'$
and $\R P^7\to *$: 
$$
X = T \sqcup_{\partial T= \RP^7} \rc \RP^7.
$$
This is  a CS set with two singular strata, $S^4$ and the apex $\tv$ of the cone.
Their links are $\RP^3$ and  $\RP^7$, respectively.
 We use the  perversity $ \ov p$ defined by  $D\ov p (S^4) = 0$ and $D\ov p(\tv) = 2$.
 The hypotheses of \thmref{thm:hurintersection}.2) are not satisfied since
$
 \pi^{\ov p}_1(\RP^7) =  \pi^{\ov p}_1(\RP^3) = \Z_2\ne 0
 $.
 
 Let $0<j\leq 2$.
 For the determination of $\pi_{j}^{\ov{p}}(X)$, we observe that the
 $\ov p$-allowability condition  \eqref{equa:admissibleST} gives
$
\pi^{\ov p}_{j} (X) \stackrel {} = \pi^{\ov p}_{j}(X\menos \{\tv\} )$.
 As $X\menos\{\tv\}$ is stratified homotopy equivalent
to $T$, we get  $\pi^{\ov p}_{j} (X)= \pi^{\ov p}_{j}(T )$.
From the homotopy long exact sequence of the bundle $\tau''$ (\cite[Theorem 2.2]{MR1404919}), 
we deduce
$\pi_{j}^{\ov p} (T)=\pi_j^{\ov p} (\rc \RP^3)=0$. 
Therefore, the condition ``$\pi_{j}^{\ov p} (X)=0$ for $j\leq 2$''  is satisfied.

We claim that $h_{3}^{\ov{p}}\colon \pi_{3}^{\ov{p}}(X )\to {H}_{3}^{\ov{p}}(X;\Z)$ is  an isomorphism
and
$h_{4}^{\ov{p}}\colon \pi_{4}^{\ov{p}}(X )\to {H}_{4}^{\ov{p}}(X;\Z)$ is not an epimorphism, which is equivalent to:
 \emph{$  \cJ^{\ov{p}}_{3} $ is  an isomorphism and $  \cJ^{\ov{p}}_{4} $ is not surjective.}
We prove that claim by using the  Mayer-Vietoris exact sequences 
(\cite[Proposition 4.1]{CST3} and    \thmref{thm:MV}) 
associated to the covering $\{X \menos \{\tv\} , X\menos S^4\}$.
Let us notice the existence of stratified homotopy equivalences,
$X\menos \{\tv\}\simeq_{s} T$,
$X\menos S^4\simeq_{s}\rc \R P^7$
and
$(X\menos \{\tv\})\cap (X\menos S^4)\simeq_{s}\R P^7$.
We compute the various homologies.\\
$\bullet$ In $T$, the perversity $\ov{p}$ is the top perversity $\ov{t}$. 
From \cite[Proposition 5.4]{CST3}, \lemref{lem:quillen} and \corref{cor:thommather}, 
we get  $H^{\ov p}_*(T) = H_* ^{\ov t}(T) = H_{*}(T)=H_{*}(S^4)$,
and $H_*(\crG_{\ov p} T)  =  H_*(\crG_{\ov t} T)  = H_*( T) =H_{*}(S^4)$. \\
$\bullet$  From  \cite[Proposition 5.2]{CST3}, we have
$ H_j^{\ov p}(\rc \RP^7) = 0 $ if $j>2$ and $H_1^{\ov p}(\rc \RP^7)  =  H_{1}( \RP^7)=\Z_{2}$.
\propref{prop:coneandhomotopy} implies that $
 \crG_{\ov p} (\rc \RP^7)
 $ is the Eilenberg-MacLane space $K(\Z_2,1)=\R P^{\infty}$.
 
 \smallskip
Thus, in low degrees, the Mayer-Vietoris sequences reduce to
$$
\scalebox{1}{
\xymatrix@C=.5cm{
  0
   \ar[r]\ar[d]&
   \Z
    \ar[r]\ar[d] &
   H_4(\crG_{\ov p}X)   \ar[r]\ar[d]^{ \cJ^{\ov{p}}_{4}}&
\Z_2
   \ar[r]\ar[d]&
\Z_2
   \ar[r]\ar[d]&
   H_3(\crG_{\ov p} X)  \ar[d]^{ \cJ^{\ov{p}}_{3}} \ar[r] &
  \ar[d] 0
 \\
0
   \ar[r]&
\Z
    \ar[r] &
   H_4^{\ov p}(X)  \ar[r] &
\Z_2
   \ar[r]&
0
   \ar[r]&
   H_3^{\ov p}(X)  \ar[r] &
0.
 }}
$$
The horizontal map $\Z_{2}\to \Z_{2}$ being an isomorphism, the map $\cJ_{3}^{\ov{p}}$ is an isomorphism and 
$\cJ_{4}^{\ov{p}}$ is not surjective.
\end{example}

\providecommand{\bysame}{\leavevmode\hbox to3em{\hrulefill}\thinspace}
\providecommand{\MR}{\relax\ifhmode\unskip\space\fi MR }
\providecommand{\MRhref}[2]{%
  \href{http://www.ams.org/mathscinet-getitem?mr=#1}{#2}
}
\providecommand{\href}[2]{#2}

%
\end{document}